\documentclass[12pt,reqno]{amsart}
\usepackage[T1]{fontenc}
\usepackage{lmodern,amsmath,amssymb,amsthm,mathtools,aliascnt}
\usepackage[marginparwidth=0pt,margin=24truemm]{geometry}
\usepackage{microtype,xcolor,graphicx}
\usepackage{enumitem}
\usepackage{float}
\usepackage{tikz}
\usetikzlibrary{arrows.meta}
\definecolor{ink}{HTML}{1F4E79}
\definecolor{fillblue}{HTML}{E8F2FA}
\definecolor{excluded}{HTML}{858E99}
\definecolor{gridline}{HTML}{D8DFE5}
\definecolor{focus}{HTML}{A34D00}

\newcommand{\latticepanel}[8]{%
 \node at (#1,47.5) {#2};
 \begin{scope}[shift={(#1,25.6)},x=5mm,y=5mm]
  \pgfmathsetmacro{\dirsign}{((#5)*(#4)-(#6)*(#3))>0 ? 1 : -1}
  \begin{scope}
   \clip (-2.75,-2.75) rectangle (2.75,2.75);
   \fill[fillblue] ({10*(#4)},{-10*(#3)}) -- ({-10*(#4)},{10*(#3)})
     -- ({-10*(#4)+10*(#3)},{10*(#3)+10*(#4)}) -- ({10*(#4)+10*(#3)},{-10*(#3)+10*(#4)}) -- cycle;
   \draw[gridline,line width=.2pt] (-2.75,-2.75) grid[step=1] (2.75,2.75);
   \draw[excluded,line width=.9pt,dash pattern=on 2.2pt off 1.8pt] (0,0) -- ({-\dirsign*10*(#4)},{\dirsign*10*(#3)});
   \draw[ink,line width=1.25pt] (0,0) -- ({\dirsign*10*(#4)},{-\dirsign*10*(#3)});
  \end{scope}
  \draw[excluded,line width=.4pt,->] (-2.95,0) -- (3.15,0) node[right] {$n$};
  \draw[excluded,line width=.4pt,->] (0,-2.95) -- (0,3.15) node[above] {$m$};
  \foreach \X in {-2,...,2} \foreach \Y in {-2,...,2} {
   \pgfmathtruncatemacro{\isorigin}{(\X==0 && \Y==0)}
   \pgfmathsetmacro{\lone}{(#3)*\X+(#4)*\Y}
   \pgfmathsetmacro{\ltwo}{(#5)*\X+(#6)*\Y}
   \pgfmathtruncatemacro{\inc}{(\lone>0) || (\lone==0 && \ltwo>0)}
   \ifnum\isorigin=0
    \ifnum\inc=1
     \draw[draw=ink,fill=ink,line width=.4pt] (\X,\Y) circle[radius=1.05pt];
    \else
     \draw[draw=excluded,fill=white,line width=.4pt] (\X,\Y) circle[radius=1.05pt];
    \fi
   \fi
  }
  \draw[draw=excluded,fill=white,line width=.55pt] (0,0) circle[radius=1.3pt];
  \node[below right,font=\scriptsize] at (0,0) {$0$};
  \node[below,font=\scriptsize,text=excluded] at (2,-.12) {$2$};
  \node[right,font=\scriptsize,text=excluded] at (.12,2) {$2$};
  \ifnum#8=1
   \draw[focus,line width=.8pt] (0,1) circle[radius=2.6pt];
   \node[right,text=focus] at (.2,1) {$\tau$};
  \fi
 \end{scope}
 \node[font=\footnotesize] at (#1,7.3) {#7};
}
\newcommand{\latticelegend}{%
 \draw[ink,line width=1.25pt] (17,1.8) -- (25,1.8);
 \node[right,font=\footnotesize] at (26,1.8) {included half-line};
 \draw[excluded,line width=.9pt,dash pattern=on 2.2pt off 1.8pt] (87,1.8) -- (95,1.8);
 \node[right,font=\footnotesize] at (96,1.8) {excluded half-line};
}
\definecolor{mylinkcolor}{RGB}{166,77,121}
\definecolor{mycitecolor}{RGB}{17,85,204}
\usepackage{hyperref}
\usepackage[nameinlink]{cleveref}
\hypersetup{colorlinks=true,linkcolor=mylinkcolor,citecolor=mycitecolor,
 urlcolor=mycitecolor,bookmarksnumbered=true,bookmarksopen=true}
\numberwithin{equation}{section}
\newtheorem{theorem}{Theorem}[section]
\newaliascnt{lemma}{theorem}
\newtheorem{lemma}[lemma]{Lemma}
\aliascntresetthe{lemma}
\newaliascnt{proposition}{theorem}
\newtheorem{proposition}[proposition]{Proposition}
\aliascntresetthe{proposition}
\newaliascnt{corollary}{theorem}
\newtheorem{corollary}[corollary]{Corollary}
\aliascntresetthe{corollary}
\theoremstyle{definition}
\newaliascnt{definition}{theorem}

\aliascntresetthe{definition}
\newaliascnt{example}{theorem}
\newtheorem{example}[example]{Example}
\aliascntresetthe{example}
\newaliascnt{remark}{theorem}
\newtheorem{remark}[remark]{Remark}
\aliascntresetthe{remark}
\crefname{theorem}{Theorem}{Theorems}
\crefname{lemma}{Lemma}{Lemmas}
\crefname{proposition}{Proposition}{Propositions}
\crefname{corollary}{Corollary}{Corollaries}
\crefname{definition}{Definition}{Definitions}
\crefname{example}{Example}{Examples}
\crefname{remark}{Remark}{Remarks}
\crefname{section}{Section}{Sections}
\crefname{subsection}{Section}{Sections}
\crefname{figure}{Figure}{Figures}
\crefname{enumi}{}{}
\setlist[enumerate]{label=(\roman*),ref=(\roman*)}
\allowdisplaybreaks[2]
\newcommand{\ZZ}{\mathbb Z}
\newcommand{\QQ}{\mathbb Q}
\newcommand{\RR}{\mathbb R}
\newcommand{\CC}{\mathbb C}
\newcommand{\HH}{\mathbb H}
\newcommand{\one}{\mathbf 1}
\newcommand{\Ab}{\mathcal A}
\newcommand{\Ef}{\mathfrak E}
\DeclareMathOperator{\Rea}{Re}
\DeclareMathOperator{\Ima}{Im}
\DeclareMathOperator{\Res}{Res}
\DeclareMathOperator{\Log}{Log}
\DeclareMathOperator{\SL}{SL}
\DeclareMathOperator{\Span}{span}

\title{Modularity criteria for multiple Eisenstein series}
\author{Jinbo Yu}
\address{Graduate School of Mathematics, Nagoya University, Nagoya, Japan.}
\email{jinbo.yu.e6@math.nagoya-u.ac.jp}
\date{\today}

\begin{document}
\begin{abstract}
We give two criteria for a linear combination of multiple Eisenstein series of fixed weight to be a modular form, one in terms of its Fourier coefficients and one in terms of its $L$-function. The first criterion follows from an exact transformation formula under $\SL_2(\ZZ)$, which describes the modular defect of multiple Eisenstein series explicitly. The second criterion is a Hecke converse theorem for multiple Eisenstein series, which shows that the functional equation of the $L$-function alone already implies modularity.
\end{abstract}
\maketitle

\section{Introduction}\label{sec:intro}

The purpose of this paper is to give two criteria for a linear combination of multiple Eisenstein series of fixed weight to be a modular form. The first criterion is given in terms of the Fourier coefficients, and the second one in terms of the functional equation of the associated $L$-function.

For integers $k_1,\dots,k_{r-1}\ge1$ and $k_r\ge2$, the multiple zeta value is defined by
\[
 \zeta(k_1,\dots,k_r)=\sum_{0<n_1<\cdots<n_r}\frac{1}{n_1^{k_1}\cdots n_r^{k_r}}\,.
\]
Multiple Eisenstein series are holomorphic functions on the upper half-plane $\HH$, defined by an analogous sum over a lattice. For $\tau\in\HH$, the lattice $\ZZ\tau+\ZZ$ is ordered by
\begin{equation}\label{eq:order}
 0\prec m\tau+n\iff m>0,\ \text{or}\ m=0\ \text{and}\ n>0,
\end{equation}
and for $\lambda,\lambda'\in\ZZ\tau+\ZZ$ we write $\lambda\prec\lambda'$ if $0\prec\lambda'-\lambda$. For $r\ge1$ and integers $k_1,\dots,k_r\ge2$, the multiple Eisenstein series of weight $k=k_1+\cdots+k_r$ and depth $r$ is
\begin{equation}\label{eq:mes}
 G_{k_1,\dots,k_r}(\tau)=\sum_{\substack{0\prec\lambda_1\prec\cdots\prec\lambda_r\\ \lambda_1,\dots,\lambda_r\in\ZZ\tau+\ZZ}}\frac{1}{\lambda_1^{k_1}\cdots\lambda_r^{k_r}}\,.
\end{equation}
The index $(k_1,\dots,k_r)$ is called \emph{admissible} if $k_r\ge3$, which ensures that \eqref{eq:mes} converges absolutely. For $k_r=2$, $G_{k_1,\dots,k_r}$ is defined by Eisenstein summation, see \eqref{eq:eisenstein-summation}.

Double Eisenstein series were introduced by Gangl, Kaneko and Zagier \cite{GKZ} in their study of double zeta values and modular forms. They computed the Fourier expansion of double Eisenstein series, whose constant term is a double zeta value. The Fourier expansion in arbitrary depth was then computed in \cite{Ba12} and has the form
\[
 G_{k_1,\dots,k_r}(\tau)=\zeta(k_1,\dots,k_r)+\sum_{n\ge1}a_nq^n,\qquad q=e^{2\pi i\tau},
\]
where every $a_n$ is a finite sum of products of multiple zeta values, powers of $2\pi i$ and multiple divisor sums (\cref{thm:fourier}). Bachmann and Tasaka \cite{BT} related this expansion to the Goncharov coproduct. The $q$-series appearing here were studied as $q$-analogues of multiple zeta values by Bachmann and K\"uhn \cite{BK}. The algebraic structure of multiple Eisenstein series has been studied further in \cite{BB,BvI,BKM,BKa}. For $k\ge1$ put
\begin{align*}
 \Ef_k&=\Span_\CC\{G_{k_1,\dots,k_r}\ :\ r\ge1,\ k_1,\dots,k_r\ge2,\ k_1+\cdots+k_r=k\},\\
 \Ef_k^{\mathrm{adm}}&=\Span_\CC\{G_{k_1,\dots,k_r}\ :\ r\ge1,\ k_1,\dots,k_{r-1}\ge2,\ k_r\ge3,\ k_1+\cdots+k_r=k\}
\end{align*}
and $\Ef_0=\CC$. The space $\Ef_k^{\mathrm{adm}}$ is spanned by the absolutely convergent multiple Eisenstein series of weight $k$. Bachmann and Kanno \cite{BKa} showed that the complex span of multiple Eisenstein series is graded by weight.
They also showed that the $\QQ$-algebra spanned by $1$ and all multiple Eisenstein series is a polynomial algebra in $G_2$ over the subalgebra spanned by $1$ and the absolutely convergent ones.

In depth one, $G_k$ is a modular form of weight $k$ for $\SL_2(\ZZ)$ if $k\ge4$ is even. The series $G_2$ is a quasimodular form in the sense of Kaneko and Zagier \cite{KZ}. For odd $k$ the series $G_k$ is not modular. Kaneko \cite{Ka04,Ka07} studied the imaginary parts of the Fourier coefficients of the normalized double Eisenstein series $(2\pi i)^{-k}G_{k_1,k_2}$ of even weight $k$. These imaginary parts are related to period polynomials. Kaneko and Tasaka \cite{KT} and Kina \cite{Ki} studied the analogous question for double Eisenstein series of level $2$ and $4$. Tasaka \cite{Ta} gave an explicit formula which expresses a Hecke eigenform as a linear combination of double Eisenstein series.

In higher depth a multiple Eisenstein series is in general not modular \cite{Ba12}, but certain linear combinations are. For example, if $k_1,\dots,k_r\ge4$ are even, then the sum of the multiple Eisenstein series $G_{k_1',\dots,k_r'}$ over all permutations $(k_1',\dots,k_r')$ of $(k_1,\dots,k_r)$ is a modular form \cite[Satz~4.2.2]{Ba12}. Bachmann \cite{Ba12} applied Kaneko's method to double and triple Eisenstein series. This gave linear combinations of single, double and triple Eisenstein series of weight $12$ with real Fourier coefficients which are cusp forms \cite[Beispiel~4.9.8]{Ba12}. No linear combination with real Fourier coefficients which is not modular was found there \cite[Bemerkung~4.9.9]{Ba12}. In \cite{Ba20} Bachmann asked which linear combinations of multiple Eisenstein series are modular forms, and whether there is an explicit way to describe the modular defect.

Let $M_k$ be the space of modular forms of weight $k$ for $\SL_2(\ZZ)$. It is given by
\[
 M_k=\Span_\CC\{G_4^aG_6^b\ :\ a,b\ge0,\ 4a+6b=k\}.
\]
For $\mathbb A=\QQ$ or $\RR$, a superscript $\mathbb A$ indicates that all Fourier coefficients lie in $\mathbb A$. Identifying functions with their Fourier expansions, we put
\[
 \Ef_k^{\mathbb A,\mathrm{adm}}=\Ef_k^{\mathrm{adm}}\cap\mathbb A[[q]],\qquad M_k^{\mathbb A}=M_k\cap\mathbb A[[q]].
\]
Our first main result answers the first question on $\Ef_k^{\mathrm{adm}}$ in each fixed weight $k$. \Cref{thm:transformation} below describes the modular defect of real linear combinations explicitly.

\begin{theorem}\label{thm:complex}
Let $k\ge1$. Then
\(
 \Ef_k^{\RR,\mathrm{adm}}=M_k^\RR\text{ and }\Ef_k^{\QQ,\mathrm{adm}}=M_k^\QQ.
\)
In particular, a complex linear combination of admissible multiple Eisenstein series of weight $k$ is a modular form if all its Fourier coefficients are real. For odd $k$ such a linear combination vanishes.
\end{theorem}

It shows that the observation in \cite[Bemerkung~4.9.9]{Ba12} holds for $\Ef_k^{\mathrm{adm}}$ for every $k$.

\begin{remark}\label{rem:F8}
The weight in \cref{thm:complex} is fixed. This is necessary. Put $\Ef^{\mathrm{adm}}=\bigoplus_{k\ge0}\Ef_k^{\mathrm{adm}}$ with $\Ef_0^{\mathrm{adm}}=\CC$. By \cite{BKa}, this is an algebra, and so is $\Ef^{\mathrm{adm}}\cap\RR[[q]]$. We have
\[
 \bigoplus_{k\ge0}M_k^\RR=\bigoplus_{k\ge0}\Ef_k^{\RR,\mathrm{adm}}\subsetneq\Ef^{\mathrm{adm}}\cap\RR[[q]].
\]
For example, $i(G_3-\zeta(3))$ has real Fourier coefficients by \cref{ex:depth-one}, but its component $iG_3$ of weight $3$ does not. The same holds for $\Ef^{\mathrm{adm}}\cap\QQ[[q]]$ and $(2\pi i)^{-3}(G_3-\zeta(3))$.

For the whole space $\Ef_k$ we compare with quasimodular forms. Let 
\[
 \widetilde M_k=\Span_\CC\{G_2^aG_4^bG_6^c\ :\ a,b,c\ge0,\ 2a+4b+6c=k\}
\]
be the space of quasimodular forms of weight $k$.
Put
\[
 \Ef_k^{\mathbb A}=\Ef_k\cap\mathbb A[[q]],\qquad \widetilde M_k^{\mathbb A}=\widetilde M_k\cap\mathbb A[[q]].
\]
$\widetilde M_k$ is contained in $\Ef_k$ and therefore $\widetilde M_k^{\mathbb A}\subset\Ef_k^{\mathbb A}$. For example, for even $k\ge4$ the quasimodular form $q\frac{d}{dq}G_{k-2}$ has real Fourier coefficients and is not modular. By \cref{thm:complex} it does not lie in $\Ef_k^{\mathrm{adm}}$. So the assumption that the indices are admissible cannot be omitted in \cref{thm:complex}. 

A natural guess is that $\Ef_k^{\mathbb A}=\widetilde M_k^{\mathbb A}$. This is not true. In fact, besides the quasimodular forms, $\Ef_k$ contains other functions with real Fourier coefficients. For example, the Fourier expansion in \cite{Ba12} shows that
\(
 F_8=3G_{2,2,4}-6G_{2,3,3}-2G_{2,4,2}-4G_{3,2,3}-8G_{3,3,2}
\)
has real Fourier coefficients. More precisely, $(2\pi i)^{-8}F_8$ has rational Fourier coefficients, and
\[
 (2\pi i)^{-8}F_8=-\frac{1}{3110400}-\frac{q}{405}-\frac{13q^2}{2160}+\frac{122q^3}{405}+\frac{10787q^4}{6480}+O(q^5).
\]
But $F_8$ is not quasimodular. Indeed, $\widetilde M_8$ is spanned by $G_2^4$, $G_2^2G_4$, $G_2G_6$ and $G_4^2$. The first five Fourier coefficients above are not those of a linear combination of these four series. One can show that $\Ef_8^\QQ$ is spanned by $\widetilde M_8^\QQ$ and $(2\pi i)^{-8}F_8$. It would be interesting to determine the algebra $\bigoplus_{k\ge0}\Ef_k^\QQ$. By \cite{BKM}, it is closed under $q\frac{d}{dq}$.

Bachmann and Burmester \cite{BB2} introduce a coaction on the space of formal multiple Eisenstein series. With this coaction they obtain a criterion for quasimodularity. This criterion asks that a certain part of the coaction vanishes. This part is closely related to the imaginary parts of the Fourier coefficients. It would be interesting to see whether the criterion in \cref{thm:complex} is related to this coaction.
\end{remark}

\Cref{thm:complex} follows from the next theorem together with the weight grading in \cite{BKa}. The next theorem concerns real linear combinations. It also contains the direction from modularity to real Fourier coefficients.

\begin{theorem}\label{thm:main}
Let $F$ be a real linear combination of admissible multiple Eisenstein series of weight $k$, and let $F(\tau)=\sum_{n\ge0}a_nq^n$ be its Fourier expansion.
\begin{enumerate}
\item If $k$ is even, then $F$ is a modular form of weight $k$ for $\SL_2(\ZZ)$ if and only if $a_n\in\RR$ for all $n\ge0$.
\item If $k$ is odd and $a_n\in i\RR$ for all $n\ge0$, then $F=0$.
\end{enumerate}
\end{theorem}

\Cref{thm:main} and \cite[Main Theorem~B(ii)]{BKa} give the following arithmetic statement.

\begin{corollary}\label{cor:rational}
Let $G_1,\dots,G_N$ be all admissible multiple Eisenstein series of weight $k$.
\begin{enumerate}
\item The vectors $(c_1,\dots,c_N)\in\CC^N$ with $\sum_\nu c_\nu G_\nu\in M_k$ form a vector space with a basis of rational vectors.
\item Let $(c_1,\dots,c_N)\in\QQ^N$ and $F=\sum_\nu c_\nu G_\nu$. Then $F$ is a modular form if and only if $(2\pi i)^{-k}F$ has real Fourier coefficients. In this case all Fourier coefficients of $(2\pi i)^{-k}F$ are rational.
\end{enumerate}
\end{corollary}

By (ii), every rational vector in (i) gives a modular form with Fourier coefficients in $(2\pi i)^k\QQ$. So the only irrational factor is $\pi^k$.

\Cref{thm:main} follows from an exact transformation formula. For an integer $k$ and a function $F\colon\HH\to\CC$ put
\begin{equation}\label{eq:flat}
 F^\flat(\tau)=(-1)^k\,\overline{F(-\bar\tau)}.
\end{equation}
If $F(\tau)=\sum_{n\ge0}a_nq^n$, then $F^\flat(\tau)=\sum_{n\ge0}(-1)^k\overline{a_n}\,q^n$. Thus for even $k$ the function $F-F^\flat$ is formed by the imaginary parts of the Fourier coefficients of $F$. For a real linear combination $F$ of admissible multiple Eisenstein series of weight $k$, the function $F^\flat$ is obtained from \eqref{eq:mes} by reversing the order on each horizontal line $m\tau+\ZZ$ of the lattice (\cref{lem:change}). For $\gamma\in\SL_2(\ZZ)$ with bottom row $(c,d)$, the slash action is $(F|_k\gamma)(\tau)=(c\tau+d)^{-k}F(\gamma\tau)$. For coprime integers $c\ge1$ and $d$, let $\gamma_{c,d}\in\SL_2(\ZZ)$ denote any matrix with bottom row $(c,d)$.

\begin{theorem}\label{thm:transformation}
Let $F$ be a real linear combination of admissible multiple Eisenstein series of weight $k$, and let $\gamma\in\SL_2(\ZZ)$ have bottom row $(c,d)$ with $c\ge1$. Then
\[
 F|_k\gamma-(-1)^kF=\sum_{(c',d')}(F-F^\flat)|_k\gamma_{c',d'},
\]
where $(c',d')$ runs over all pairs of coprime integers with $c'\ge1$ and $d'/c'\le d/c$. The series converges absolutely at every point of $\HH$, and its terms do not depend on the choice of $\gamma_{c',d'}$.
\end{theorem}

For $\gamma=S=\left(\begin{smallmatrix}0&-1\\1&0\end{smallmatrix}\right)$ this reads
\begin{equation}\label{eq:S-formula}
 F|_kS-(-1)^kF=\sum_{\substack{c\ge1,\ d\le0\\ \gcd(c,d)=1}}(F-F^\flat)|_k\gamma_{c,d}.
\end{equation}
If $k$ is even and all Fourier coefficients of $F$ are real, then $F^\flat=F$, and \eqref{eq:S-formula} gives $F|_kS=F$. This is the main step in the proof of \cref{thm:main}.

The proof of \cref{thm:transformation} starts from an observation of \cite{Ba12,BKa}: $F|_k\gamma$ is again a lattice sum of the form \eqref{eq:mes} with respect to another order of the lattice. For $\gamma=S$ the positive lattice points of the new order are the points $m\tau+n$ with $n<0$ and the points $m\tau$ with $m>0$. This order is connected with the opposite of \eqref{eq:order} by a family of orders whose half-planes rotate with a rational parameter $x$ (\cref{fig:rotation}). The lattice sum changes only when the boundary line of the half-plane passes through lattice points, and these changes are the terms of the series in \cref{thm:transformation}. The main point is that the changes are absolutely summable. For indices with an entry equal to $2$, this uses a cancellation between translated lattice sums (\cref{lem:blocks}).

\begin{example}\label{ex:depth-one}
Let $k\ge3$ be odd. The Fourier expansion
\[
 G_k(\tau)=\zeta(k)+\frac{(-2\pi i)^k}{(k-1)!}\sum_{n\ge1}\sigma_{k-1}(n)q^n,
\]
where $\sigma_{k}(n)$ is the divisor function, has a real constant term and purely imaginary coefficients otherwise. Hence $G_k-G_k^\flat=2\zeta(k)$, and \eqref{eq:S-formula} becomes
\[
 G_k|_kS+G_k=2\zeta(k)\sum_{\substack{c\ge1,\ d\le0\\ \gcd(c,d)=1}}\frac{1}{(c\tau+d)^k}\,.
\]
\end{example}

The second criterion concerns $L$-functions. For a function $F(\tau)=\sum_{n\ge0}a_nq^n$ put
\[
 L(F,s)=\sum_{n\ge1}\frac{a_n}{n^s},\qquad \Lambda_F(s)=(2\pi)^{-s}\Gamma(s)L(F,s).
\]
For $F\in M_k$, Hecke \cite{He} showed that $\Lambda_F(s)+a_0/s+i^ka_0/(k-s)$ extends to an entire function which is bounded in vertical strips. Moreover $\Lambda_F(s)=i^k\Lambda_F(k-s)$. Conversely, if the coefficients $a_n$ grow at most polynomially and $\Lambda_F$ has these properties, then $F\in M_k$ \cite{He}. Weil \cite{We} extended this converse theorem to congruence subgroups by using the functional equations of twists. Converse theorems for quasimodular forms were proved by Bhand and Shankhadhar \cite{BS} and by Charan, Meher, Shankhadhar and Singh \cite{CMSS}. We study the $L$-functions of multiple Eisenstein series.

\begin{proposition}\label{prop:analytic}
Let $F\in\Ef_k$.
\begin{enumerate}
\item The series $L(F,s)$ converges absolutely for $\Rea s>k$.
\item The function $\Lambda_F$ has a meromorphic continuation to $\CC$, and all its poles are real.
\item For real numbers $\sigma'<\sigma''$ there is $\kappa>0$ such that
\[
 \Lambda_F(s)\ll(1+|\Ima s|)^\kappa e^{-\pi|\Ima s|/2}\qquad(\sigma'\le\Rea s\le\sigma'',\ |\Ima s|\ge1).
\]
\end{enumerate}
\end{proposition}

\Cref{prop:analytic} follows from the Fourier expansion together with results of Sargos \cite{Sa} on Dirichlet series attached to polynomials. However, $\Lambda_F$ may have poles in $(0,k)$. For example, $\Lambda_{G_k}$ has a pole at $s=1$ if $k\ge3$ is odd. Therefore Hecke's converse theorem does not apply directly. Our second main result shows that for multiple Eisenstein series of even weight at least four the functional equation alone implies modularity.

\begin{theorem}\label{thm:hecke}
Let $k\ge4$ be even and let $F\in\Ef_k$. Then $F\in M_k$ if and only if
\[
 \Lambda_F(s)=i^k\Lambda_F(k-s)
\]
as meromorphic functions on $\CC$.
\end{theorem}

\Cref{thm:complex,thm:main,thm:transformation} and \cref{cor:rational} concern admissible indices. In contrast, \cref{thm:hecke} applies to the whole space $\Ef_k$. It allows indices with $k_r=2$. In the proof, the functional equation and Mellin inversion show that on $\HH$ the function $\psi(z)=F(z)-z^{-k}F(-1/z)$ is a finite linear combination of the functions $z^{-\rho}(\Log z)^j$ with $0<\rho<k$ and $j\ge0$. Here $\Log$ is the principal branch of the logarithm. These terms come from the poles of $\Lambda_F$ in $(0,k)$. Bochner \cite{Bo} showed that a functional equation with singularities in a bounded set is equivalent to a transformation law with an additional residual term. Knopp \cite{Kn81} proved such a Hecke correspondence for modular integrals whose rational period function has poles only at $0$ and $\infty$, see also \cite{Kn78}. Hawkins and Knopp \cite{HK} treated general rational period functions for the theta group. Culp-Ressler \cite{CR} treated general rational period functions for $\SL_2(\ZZ)$. Hassen \cite{Has} studied log-polynomial period functions, which are sums of terms $z^{-\rho}(\Log z)^j$. In our situation the three-term relation
\[
 \psi(z)=\psi(z+1)+(z+1)^{-k}\psi\Big(\frac{z}{z+1}\Big),
\]
which follows from $F(z+1)=F(z)$, forces $\psi=0$ when $k\ge4$. Three-term relations of this type also appear in the theory of period functions of Lewis and Zagier \cite{LZ}, see \cite[Theorem~1]{BC} for Eisenstein series. The condition $k\ge4$ cannot be omitted, since $\Lambda_{G_2}(s)=-\Lambda_{G_2}(2-s)$ but $G_2$ is not a modular form (\cref{rem:G2}).

The contents of this paper are as follows. In \cref{sec:mes} we recall the definition and the Fourier expansion of multiple Eisenstein series. In \cref{sec:orders} we express the slash action on admissible multiple Eisenstein series by a change of the order of summation. \Cref{sec:transformation} contains the proofs of \cref{thm:transformation,thm:main,thm:complex} and \cref{cor:rational}. The main step is \cref{prop:jumps}, which describes how the lattice sums depend on the rotating order. In \cref{sec:hecke} we study the $L$-functions of multiple Eisenstein series and prove \cref{thm:hecke}.

\section*{Acknowledgement}
The author would like to thank Henrik Bachmann for helpful suggestions and comments. The author was financially supported by JST SPRING, Grant Number JPMJSP2125.

\section{Preliminaries}\label{sec:mes}

In this section we recall the definition and the Fourier expansion of multiple Eisenstein series, and some facts on modular forms. Throughout, $\tau\in\HH$ and $q=e^{2\pi i\tau}$.

An \emph{index} is a tuple $\Bbbk=(k_1,\dots,k_r)$ of integers $k_1,\dots,k_r\ge2$ with $r\ge1$. Its weight is $k_1+\cdots+k_r$ and its depth is $r$. For such an index put $G_\Bbbk=G_{k_1,\dots,k_r}$ and $\zeta(\Bbbk)=\zeta(k_1,\dots,k_r)$. For an admissible index, $G_\Bbbk$ is defined by the absolutely convergent series \eqref{eq:mes}, see \cref{lem:convergence}. For an arbitrary index, $G_\Bbbk$ is defined by Eisenstein summation:
\begin{equation}\label{eq:eisenstein-summation}
 G_\Bbbk(\tau)=\lim_{M\to\infty}\lim_{N\to\infty}
 \sum_{\substack{0\prec\lambda_1\prec\cdots\prec\lambda_r\\ \lambda_i=m_i\tau+n_i,\ |m_i|\le M,\ |n_i|\le N}}
 \frac{1}{\lambda_1^{k_1}\cdots\lambda_r^{k_r}}\,.
\end{equation}
For admissible indices this agrees with \eqref{eq:mes}.

For an index $\Bbbk$ of weight $k$ put
\begin{equation}\label{eq:g}
 g_\Bbbk(q)=\frac{(-2\pi i)^k}{(k_1-1)!\cdots(k_r-1)!}
 \sum_{0<u_1<\cdots<u_r}\ \sum_{v_1,\dots,v_r\ge1}
 v_1^{k_1-1}\cdots v_r^{k_r-1}q^{u_1v_1+\cdots+u_rv_r}.
\end{equation}
The $q$-series $(-2\pi i)^{-k}g_\Bbbk(q)$ has nonnegative rational coefficients, which grow at most polynomially. In particular $g_\Bbbk(q)$ converges for $|q|<1$.

\begin{theorem}[{\cite{Ba12}, \cite[Propositions~2.2 and~2.5]{BT}}]\label{thm:fourier}
Let $\Bbbk$ be an index of weight $k$. Then
\[
 G_\Bbbk(\tau)=\zeta(\Bbbk)+\sum_{\Bbbk'}c_{\Bbbk'}\,g_{\Bbbk'}(q)+g_\Bbbk(q).
\]
Here $\Bbbk'$ runs over finitely many indices of weight less than $k$. Each $c_{\Bbbk'}$ is a $\QQ$-linear combination of multiple zeta values of weight $k$ minus the weight of $\Bbbk'$.
\end{theorem}

In particular, $G_\Bbbk$ is holomorphic on $\HH$ and satisfies $G_\Bbbk(\tau+1)=G_\Bbbk(\tau)$. It is bounded on $\{\tau\in\HH:\Ima\tau\ge1\}$, and its constant term is $\zeta(\Bbbk)$.

For an integer $k$, a function $F\colon\HH\to\CC$ and $\gamma=\left(\begin{smallmatrix}a&b\\c&d\end{smallmatrix}\right)\in\SL_2(\ZZ)$ put
\[
 (F|_k\gamma)(\tau)=(c\tau+d)^{-k}F\Big(\frac{a\tau+b}{c\tau+d}\Big).
\]
This is a right action of $\SL_2(\ZZ)$, and $F|_k(-I)=(-1)^kF$. A modular form of weight $k$ for $\SL_2(\ZZ)$ is a holomorphic function $F$ on $\HH$ with $F|_k\gamma=F$ for all $\gamma\in\SL_2(\ZZ)$, which is bounded on $\{\Ima\tau\ge1\}$. Let $M_k$ be the space of these forms. Since $\SL_2(\ZZ)$ is generated by
\[
 S=\begin{pmatrix}0&-1\\1&0\end{pmatrix}\qquad\text{and}\qquad T=\begin{pmatrix}1&1\\0&1\end{pmatrix},
\]
a holomorphic function $F$ on $\HH$ which is bounded on $\{\Ima\tau\ge1\}$ lies in $M_k$ if and only if $F|_kT=F$ and $F|_kS=F$. For odd $k$ we have $M_k=\{0\}$, since $F=F|_k(-I)=-F$ for $F\in M_k$.

For an integer $k$ and $F\colon\HH\to\CC$, the function $F^\flat$ is defined by \eqref{eq:flat}. In this paper $k$ is always the weight of $F$. Since $\overline{e^{2\pi in(-\bar\tau)}}=e^{2\pi in\tau}$, the Fourier expansion $F(\tau)=\sum_{n\ge0}a_nq^n$ gives
\begin{equation}\label{eq:flat-fourier}
 F^\flat(\tau)=\sum_{n\ge0}(-1)^k\,\overline{a_n}\,q^n.
\end{equation}
Therefore, for even $k$ we have $F^\flat=F$ if and only if all $a_n$ are real. For odd $k$ we have $F^\flat=F$ if and only if all $a_n$ are purely imaginary. The map $F\mapsto F^\flat$ is conjugate linear.

\section{Changing the order of summation}\label{sec:orders}

In this section we describe the slash action on admissible multiple Eisenstein series by a change of the order of summation. This idea goes back to \cite{Ba12}.

A subset $C\subset\ZZ^2\setminus\{(0,0)\}$ is called a \emph{positive set} if $\ZZ^2\setminus\{(0,0)\}$ is the disjoint union of $C$ and $-C$, and $C+C\subset C$. For $\tau\in\HH$ we identify $(m,n)\in\ZZ^2$ with the lattice point $m\tau+n$, and we write $m\tau+n\in C$ if $(m,n)\in C$. Then
\[
 \lambda\prec_C\lambda'\iff\lambda'-\lambda\in C\qquad(\lambda,\lambda'\in\ZZ\tau+\ZZ)
\]
defines a total order on $\ZZ\tau+\ZZ$ which is invariant under translations, and $\preceq_C$ denotes the corresponding non-strict order. Put
\begin{align*}
 P&=\{(m,n):\ m>0,\ \text{or}\ m=0\ \text{and}\ n>0\},\\
 P^\flat&=\{(m,n):\ m>0,\ \text{or}\ m=0\ \text{and}\ n<0\},\\
 C_x&=\{(m,n):\ n+xm<0,\ \text{or}\ n+xm=0\ \text{and}\ m>0\},\\
 C_x^+&=\{(m,n):\ n+xm<0,\ \text{or}\ n+xm=0\ \text{and}\ m<0\},
\end{align*}
where $x\in\QQ$. The set $P$ is the positive set of the order \eqref{eq:order}.

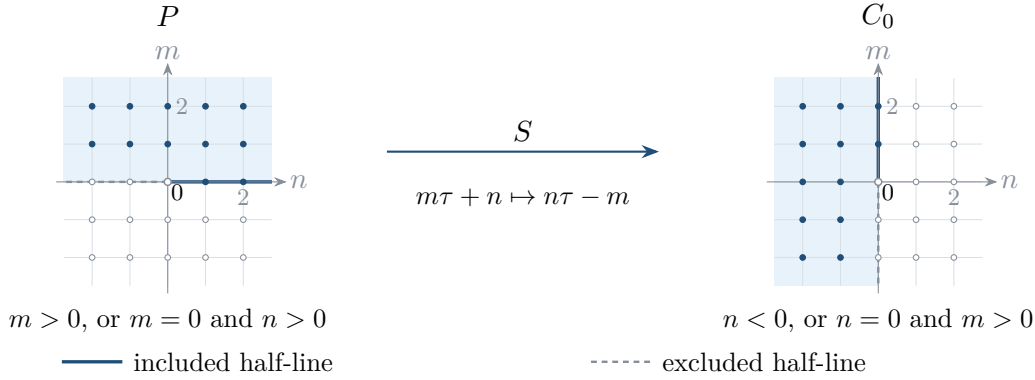
\begin{figure}[H]
 \centering
 \begin{tikzpicture}[x=1mm,y=1mm,>=Stealth,every node/.style={font=\small,inner sep=1.1pt}]
  \path[use as bounding box] (0,0) rectangle (156,49.5);
  \latticepanel{31}{$P$}{0}{1}{1}{0}{$m>0$, or $m=0$ and $n>0$}{0}
  \latticepanel{125}{$C_0$}{-1}{0}{0}{1}{$n<0$, or $n=0$ and $m>0$}{0}
  \draw[->,ink,line width=.8pt] (60,29.6) -- (96,29.6);
  \node[above] at (78,30.6) {$S$};
  \node[font=\footnotesize] at (78,23.6) {$m\tau+n\mapsto n\tau-m$};
  \latticelegend
 \end{tikzpicture}
 \caption{The positive sets $P$ and $C_0$. The horizontal axis is $n$ and the vertical axis is $m$, so that a point $(n,m)$ represents $m\tau+n$. Filled points belong to the set and open points do not. The open half-plane is shaded, and the solid half-line of its boundary belongs to the set.}
 \label{fig:S}
\end{figure}

Let $L_1,L_2\colon\RR^2\to\RR$ be linear maps such that $L_1(m,n)=L_2(m,n)=0$ only for $(m,n)=(0,0)$. The set
\begin{equation}\label{eq:lex}
 \{(m,n)\ne(0,0):\ L_1(m,n)>0,\ \text{or}\ L_1(m,n)=0\ \text{and}\ L_2(m,n)>0\}
\end{equation}
is the set of positive elements for the lexicographic order of the pairs $(L_1(m,n),L_2(m,n))$, so it is a positive set. Geometrically, it consists of the lattice points in the open half-plane $L_1>0$ and on one of the two open half-lines of the line $L_1=0$. The five sets above are of the form \eqref{eq:lex}. For such a set $C$ we write $L_1^C$ and $L_2^C$ for its two maps, so that
\begin{gather*}
 L_1^{P}(m,n)=L_1^{P^\flat}(m,n)=m,\qquad
 L_1^{-P}(m,n)=-m,\qquad
 L_1^{C_x}(m,n)=L_1^{C_x^+}(m,n)=-n-xm,\\
 L_2^{P}(m,n)=n,\qquad
 L_2^{P^\flat}(m,n)=L_2^{-P}(m,n)=-n,\qquad
 L_2^{C_x}(m,n)=m,\qquad
 L_2^{C_x^+}(m,n)=-m.
\end{gather*}
\Cref{fig:S} shows $P$ and $C_0$, and \cref{fig:rotation} shows $C_0$, $C_0^+$ and $C_1$. As $x$ increases, the boundary line $n+xm=0$ of $C_x$ rotates, and $C_x$ tends to $-P$ as $x\to+\infty$ (\cref{lem:indicator}).

\begin{figure}[H]
 \centering
 \begin{tikzpicture}[x=1mm,y=1mm,>=Stealth,every node/.style={font=\small,inner sep=1.1pt}]
  \path[use as bounding box] (0,0) rectangle (156,49.5);
  \latticepanel{26}{$C_0$}{-1}{0}{0}{1}{$n<0$, or $n=0$ and $m>0$}{1}
  \latticepanel{78}{$C_0^{+}$}{-1}{0}{0}{-1}{$n<0$, or $n=0$ and $m<0$}{1}
  \latticepanel{130}{$C_1$}{-1}{-1}{0}{1}{$n+m<0$, or $n+m=0$ and $m>0$}{0}
  \latticelegend
 \end{tikzpicture}
 \caption{The positive sets $C_0$, $C_0^+$ and $C_1$, drawn as in \cref{fig:S}. The sets $C_0$ and $C_0^+$ differ only on the line $n=0$. For example, $\tau\in C_x$ for $x\le0$ and $\tau\notin C_x$ for $x>0$, while $\tau\in C_0$ and $\tau\notin C_0^+$.}
 \label{fig:rotation}
\end{figure}
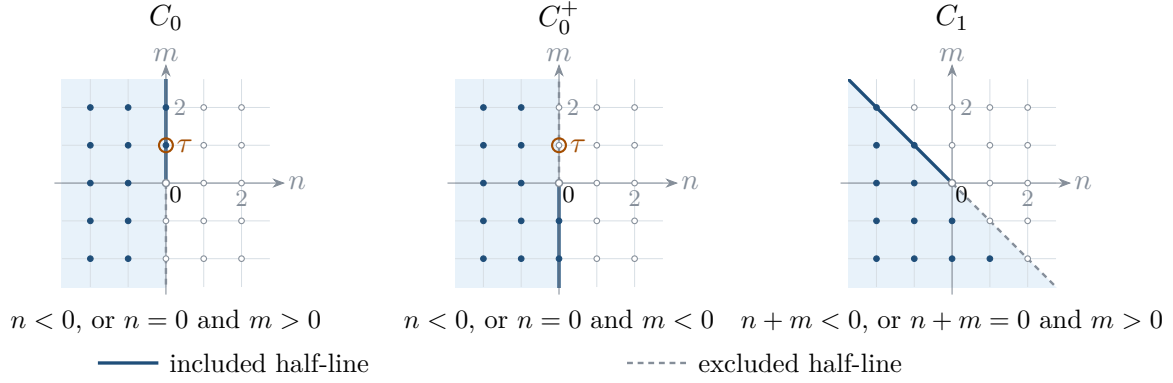

Let $C$ be a positive set of the form \eqref{eq:lex}, and let $\Bbbk=(k_1,\dots,k_r)$ be an admissible index. A \emph{chain} for $C$ is a tuple $(\lambda_1,\dots,\lambda_r)$ of lattice points with $0\prec_C\lambda_1\prec_C\cdots\prec_C\lambda_r$. Put
\begin{equation}\label{eq:GC}
 G_\Bbbk^C(\tau)=\sum_{0\prec_C\lambda_1\prec_C\cdots\prec_C\lambda_r}\frac{1}{\lambda_1^{k_1}\cdots\lambda_r^{k_r}}\,,
\end{equation}
where the sum runs over all chains for $C$. This series converges absolutely by \cref{lem:convergence}, and $G_\Bbbk^P=G_\Bbbk$. Since the sets $C$ are defined by conditions on the integer coordinates, $G_\Bbbk^C$ is a function on $\HH$.

For $\tau\in\HH$ let $\delta=\delta(\tau)>0$ be the minimum of $|\lambda|$ over all $\lambda\in\ZZ\tau+\ZZ$ with $\lambda\ne0$. Let $C$ have the form \eqref{eq:lex}. For $z=m\tau+n\in\CC$ with $m,n\in\RR$ put $L_1^C(z)=L_1^C(m,n)$. Since $(m,n)\mapsto m\tau+n$ is an $\RR$-linear bijection from $\RR^2$ to $\CC$, this defines an $\RR$-linear map $L_1^C\colon\CC\to\RR$. If $\alpha\preceq_C\lambda\preceq_C\beta$, then $\lambda-\alpha$ and $\beta-\lambda$ lie in $C\cup\{0\}$, hence $L_1^C(\alpha)\le L_1^C(\lambda)\le L_1^C(\beta)$. Thus
\begin{equation}\label{eq:interval-strip}
 \{\lambda:\alpha\preceq_C\lambda\preceq_C\beta\}\subset\{z\in\CC:\ L_1^C(\alpha)\le L_1^C(z)\le L_1^C(\beta)\}.
\end{equation}
The right side of \eqref{eq:interval-strip} is the closed strip bounded by the parallel lines $L_1^C(z)=L_1^C(\alpha)$ and $L_1^C(z)=L_1^C(\beta)$, which pass through $\alpha$ and $\beta$. Hence its width is at most $|\beta-\alpha|$.

\begin{lemma}\label{lem:lattice}
Let $\tau\in\HH$. The following statements hold, where $\lambda$ runs through $\ZZ\tau+\ZZ$ and the implied constants depend only on $\tau$, $\kappa$ and $\mu$.
\begin{enumerate}
\item For $\kappa>2$ and $\mu\ge0$, $\displaystyle\sum_{\lambda\ne0}|\lambda|^{-\kappa}\big(\log(3+|\lambda|)\big)^\mu<\infty$.
\item For $R\ge1$, $\displaystyle\sum_{|\lambda|\ge R}|\lambda|^{-3}\ll R^{-1}$.
\item For a closed strip $D\subset\CC$ of width $w\ge0$, $\displaystyle\sum_{\lambda\in D\setminus\{0\}}|\lambda|^{-2}\ll\log(3+w)$.
\end{enumerate}
\end{lemma}

\begin{proof}
Since $|\lambda-\lambda'|\ge\delta$ for distinct lattice points $\lambda,\lambda'$, by the definition of $\delta=\delta(\tau)$, the open discs of radius $\delta/2$ around the lattice points are pairwise disjoint. Therefore the number of lattice points in a subset of $\CC$ is at most $4/(\pi\delta^2)$ times the area of the $\delta/2$-neighbourhood of the subset. Let $R\ge\delta$. The set $\{R\le|\lambda|<2R\}$ lies in the disc of radius $2R$, and $D\cap\{R\le|\lambda|<2R\}$ lies in a rectangle with side lengths $w$ and $4R$. The $\delta/2$-neighbourhoods of these sets have area at most $\pi(3R)^2$ and $(w+\delta)(4R+\delta)\le5R(w+\delta)$. Hence
\[
 \#\{\lambda:R\le|\lambda|<2R\}\ll R^2,\qquad \#\{\lambda\in D:R\le|\lambda|<2R\}\ll(1+w)R.
\]
Put $R_j=\delta2^j$ for $j\ge0$. The ranges $R_j\le|\lambda|<R_{j+1}$ cover all $\lambda\ne0$, and in the $j$-th range $|\lambda|^{-\kappa}\le R_j^{-\kappa}$ and $\log(3+|\lambda|)\ll1+j$. Hence
\[
 \sum_{\lambda\ne0}|\lambda|^{-\kappa}\big(\log(3+|\lambda|)\big)^\mu
 \ll\sum_{j\ge0}R_j^{2-\kappa}(1+j)^\mu
 \ll\sum_{j\ge0}2^{(2-\kappa)j}(1+j)^\mu<\infty,
\]
which proves (i). In (ii) only the ranges with $R_{j+1}>R$ contribute. Let $j_0$ be the smallest such $j$. Then $R/2<R_{j_0}\le R$, since $R_0=\delta\le1\le R$, and
\[
 \sum_{|\lambda|\ge R}|\lambda|^{-3}\ll\sum_{j\ge j_0}R_j^2R_j^{-3}=\sum_{j\ge j_0}\frac{1}{R_j}=\frac{2}{R_{j_0}}<\frac{4}{R}\,.
\]
In (iii) the $j$-th range contributes $\ll\min\big(R_j^2,(1+w)R_j\big)R_j^{-2}\ll\min\big(1,(1+w)2^{-j}\big)$, and
\begin{align*}
 \sum_{\lambda\in D\setminus\{0\}}|\lambda|^{-2}&\ll\sum_{j\ge0}\min\big(1,(1+w)2^{-j}\big)
 \le\#\{j\ge0:2^j\le1+w\}+(1+w)\sum_{2^j>1+w}2^{-j}\\
 &\le\log_2(1+w)+3\ll\log(3+w).\qedhere
\end{align*}
\end{proof}

\begin{lemma}\label{lem:convergence}
Let $\Bbbk$ be an admissible index and let $C$ be a positive set of the form \eqref{eq:lex}. Then the series \eqref{eq:GC} converges absolutely.
\end{lemma}

\begin{proof}
Fix the last point $\lambda_r$. By \eqref{eq:interval-strip}, the points $\lambda_1,\dots,\lambda_{r-1}$ lie in a strip of width at most $|\lambda_r|$. Since $|\lambda|^{-k_i}\le\delta^{2-k_i}|\lambda|^{-2}$ for $\lambda\ne0$, dropping the order between $\lambda_1,\dots,\lambda_{r-1}$ and using \cref{lem:lattice}~(iii) gives
\[
 \sum_{0\prec_C\lambda_1\prec_C\cdots\prec_C\lambda_r}\frac{1}{|\lambda_1|^{k_1}\cdots|\lambda_r|^{k_r}}
 \ll\sum_{\lambda_r\ne0}\frac{(\log(3+|\lambda_r|))^{r-1}}{|\lambda_r|^{k_r}}\,,
\]
which is finite by \cref{lem:lattice}~(i), since $k_r\ge3$.
\end{proof}

\begin{lemma}\label{lem:change}
Let $\Bbbk$ be an admissible index of weight $k$.
\begin{enumerate}
\item $G_\Bbbk^{P^\flat}=G_\Bbbk^\flat$ and $G_\Bbbk^{-P}=(-1)^kG_\Bbbk$.
\item Let $c\ge1$ and $d$ be coprime integers, put $x=-d/c$, and let $\gamma\in\SL_2(\ZZ)$ have bottom row $(c,d)$. Then
\[
 G_\Bbbk^{C_x}=G_\Bbbk|_k\gamma\qquad\text{and}\qquad G_\Bbbk^{C_x^+}=G_\Bbbk^\flat|_k\gamma.
\]
\end{enumerate}
\end{lemma}

The first identity in (ii) goes back to \cite{Ba12}, see also \cite[Lemma~3.1 and Remark~3.2]{BKa}.

\begin{proof}
All changes of the summation variables below are allowed by \cref{lem:convergence}.

(i) Since $\overline{m(-\bar\tau)+n}=-(m\tau-n)$ and $k_1+\cdots+k_r=k$, we have
\[
 G_\Bbbk^\flat(\tau)=(-1)^k\,\overline{G_\Bbbk(-\bar\tau)}
 =\sum_{0\prec\lambda_1\prec\cdots\prec\lambda_r}\ \prod_{i=1}^r(m_i\tau-n_i)^{-k_i},
\]
where $\lambda_i=m_i\tau+n_i$. The substitution $(m_i,n_i)\mapsto(m_i,-n_i)$ maps $P$ onto $P^\flat$, hence the chains for $P$ onto the chains for $P^\flat$. This proves $G_\Bbbk^\flat=G_\Bbbk^{P^\flat}$. The substitution $(m_i,n_i)\mapsto(-m_i,-n_i)$ maps the chains for $P$ onto the chains for $-P$ and multiplies every term by $(-1)^k$, which proves the second identity.

(ii) Let $\gamma=\left(\begin{smallmatrix}a&b\\c&d\end{smallmatrix}\right)$. For $m,n\in\ZZ$ we have
\[
 m\,\frac{a\tau+b}{c\tau+d}+n=\frac{(am+cn)\tau+(bm+dn)}{c\tau+d}\,.
\]
Since $k_1+\cdots+k_r=k$, the factors $c\tau+d$ cancel with the factor of the slash action, and
\[
 (G_\Bbbk|_k\gamma)(\tau)=\sum\ \prod_{i=1}^r\big((am_i+cn_i)\tau+(bm_i+dn_i)\big)^{-k_i},
\]
where the sum runs over the chains $0\prec m_1\tau+n_1\prec\cdots\prec m_r\tau+n_r$ for $P$. The map $(m,n)\mapsto(m',n')=(am+cn,bm+dn)$ is an additive bijection of $\ZZ^2$, since $\gamma$ is invertible over $\ZZ$. Since $c\ge1$ and $c(n'+xm')=cn'-dm'=-m$, the condition $m>0$ is equivalent to $n'+xm'<0$. If $m=0$, then $n'+xm'=0$ and $m'=cn$, so that $n>0$ if and only if $m'>0$. Therefore the map sends $P$ onto $C_x$ and $P^\flat$ onto $C_x^+$, and it sends the chains for $P$ onto the chains for $C_x$. This proves $G_\Bbbk|_k\gamma=G_\Bbbk^{C_x}$. The same computation applied to $G_\Bbbk^\flat=G_\Bbbk^{P^\flat}$ gives $G_\Bbbk^\flat|_k\gamma=G_\Bbbk^{C_x^+}$.
\end{proof}

\section{The transformation formula}\label{sec:transformation}

In this section we prove \cref{thm:transformation,thm:main,thm:complex} and \cref{cor:rational}. Fix $\tau\in\HH$ and regard $G_\Bbbk^{C_x}(\tau)$ as a function of $x\in\QQ$. By \cref{lem:change}, its value at $x=-d/c$ is $(G_\Bbbk|_k\gamma_{c,d})(\tau)$, and its limit as $x\to+\infty$ will be $(-1)^kG_\Bbbk(\tau)$. Here and in the following, all limits in $x$ are taken with $x\in\QQ$. For a function $h\colon\QQ\to\CC$ and $x_0\in\QQ$ we write $h(x_0^{+})=\lim_{x\to x_0^{+}}h(x)$ and $h(+\infty)=\lim_{x\to+\infty}h(x)$ whenever these limits exist.

We first introduce the class of functions of $x$ to which the lattice sums $G_\Bbbk^{C_x}(\tau)$ will belong. Let $\Ab$ be the set of functions $h\colon\QQ\to\CC$ with the following three properties. The limits $h(u^{+})$ for $u\in\QQ$ and $h(+\infty)$ exist. The jumps of $h$ are absolutely summable,
\begin{equation}\label{eq:summable}
 \sum_{u\in\QQ}|h(u)-h(u^{+})|<\infty.
\end{equation}
The function $h$ is the sum of its jumps,
\begin{equation}\label{eq:jump-sum}
 h(x)=h(+\infty)+\sum_{\substack{u\in\QQ\\ u\ge x}}\big(h(u)-h(u^{+})\big)\qquad(x\in\QQ).
\end{equation}
On an interval of $\RR$, functions of this kind are the saltus functions of bounded variation \cite{RN}. For $h\in\Ab$ put
\[
 \|h\|=|h(+\infty)|+\sum_{u\in\QQ}|h(u)-h(u^{+})|.
\]
By \eqref{eq:jump-sum}, $|h(x)|\le\|h\|$ for all $x\in\QQ$. For example, let $u_1<\dots<u_N$ be rational numbers. A function which is constant on each of the sets $x\le u_1$, $u_1<x\le u_2$, \dots, $u_{N-1}<x\le u_N$, $x>u_N$ belongs to $\Ab$. The lattice sums may have jumps at all rational numbers.

\Cref{ex:depth-one} is the case of depth one, where the three properties can be verified directly. Let $k\ge3$ and $x\in\QQ$, and let $\lambda=m\tau+n$ be a lattice point. If $m>0$, then $\lambda\in C_x$ if and only if $x\le-n/m$. If $m<0$, then $\lambda\in C_x$ if and only if $x>-n/m$. If $m=0$, then $\lambda\in C_x$ if and only if $n<0$. Moreover $-P=\{(m,n):\ m<0,\ \text{or}\ m=0\ \text{and}\ n<0\}$. Hence
\[
 G_k^{C_x}(\tau)-G_k^{-P}(\tau)=\sum_{\substack{m>0\\ -n/m\ge x}}\frac{1}{(m\tau+n)^k}-\sum_{\substack{m<0\\ -n/m\ge x}}\frac{1}{(m\tau+n)^k}\,,
\]
where both sums converge absolutely by \cref{lem:lattice}. We group the lattice points according to the rational number $u=-n/m$. If $u=-d/c$ with coprime integers $c\ge1$ and $d$, then the points with $-n/m=u$ are $\lambda=j(c\tau+d)$ with $j\in\ZZ$, $j\ne0$, and $m>0$ means $j>0$. Since $G_k^{-P}=(-1)^kG_k$ by \cref{lem:change}, we obtain
\begin{equation}\label{eq:depth-one-jumps}
 G_k^{C_x}(\tau)=(-1)^kG_k(\tau)+\sum_{\substack{u=-d/c\in\QQ\\ u\ge x}}\big(1-(-1)^k\big)\frac{\zeta(k)}{(c\tau+d)^k}\qquad(x\in\QQ),
\end{equation}
and the series converges absolutely. By dominated convergence, the right side tends to the sum over $u>x_0$ as $x\to x_0^{+}$, and to $(-1)^kG_k(\tau)$ as $x\to+\infty$. Thus the limits of $x\mapsto G_k^{C_x}(\tau)$ exist, its jump at $u=-d/c$ is $(1-(-1)^k)\zeta(k)(c\tau+d)^{-k}$, the jumps are absolutely summable, and \eqref{eq:depth-one-jumps} is \eqref{eq:jump-sum}. So $x\mapsto G_k^{C_x}(\tau)$ belongs to $\Ab$. Since $G_k^{C_0}=G_k|_kS$ by \cref{lem:change}, \eqref{eq:depth-one-jumps} at $x=0$ is the formula in \cref{ex:depth-one} for odd $k$, and it is $G_k|_kS=G_k$ for even $k$. In higher depth, and for entries equal to $2$, the jump at $u$ also involves lattice points off the line $cn=dm$. The rest of this section shows that the lattice sums still belong to $\Ab$.

Two classical facts concern functions on an interval of $\RR$. A saltus function is determined by its jumps up to an additive constant \cite{RN}. Functions of bounded variation form an algebra under pointwise multiplication \cite{HS}. The next lemma gives the versions on $\QQ$ that we need.

\begin{lemma}\label{lem:step}
\begin{enumerate}
\item Let $b\in\CC$, and let $b_u\in\CC$ for $u\in\QQ$ with $\sum_{u\in\QQ}|b_u|<\infty$. Then
\[
 h(x)=b+\sum_{{x\le u\in\QQ}}b_u\qquad(x\in\QQ)
\]
defines an element of $\Ab$ with $h(+\infty)=b$, $h(u)-h(u^{+})=b_u$ for all $u\in\QQ$, and $\|h\|=|b|+\sum_u|b_u|$.
\item If $h,h'\in\Ab$, then $h+h'\in\Ab$ and $hh'\in\Ab$, with $\|h+h'\|\le\|h\|+\|h'\|$ and $\|hh'\|\le\|h\|\,\|h'\|$.
\item Let $h^{(1)},h^{(2)},\ldots\in\Ab$ with $\sum_i\|h^{(i)}\|<\infty$. Then $h=\sum_ih^{(i)}$ converges uniformly on $\QQ$, $h\in\Ab$, and $\|h\|\le\sum_i\|h^{(i)}\|$. The limits of $h$ at $x_0^{+}$ and $+\infty$ are the sums of the corresponding limits of the functions $h^{(i)}$.
\end{enumerate}
\end{lemma}

\begin{proof}
(i) The series converges absolutely and uniformly in $x$. For fixed $u$, the term of the series with index $u$ is $b_u\one_{\{x\le u\}}$, where $\one_{\{x\le u\}}$ is $1$ for $x\le u$ and $0$ for $x>u$. Hence
\[
 \lim_{x\to+\infty}b_u\one_{\{x\le u\}}=0,\text{ and }
 \lim_{x\to x_0^{+}}b_u\one_{\{x\le u\}}=
 \begin{cases}
  b_u & \text{if } u>x_0,\\
  0 & \text{if } u\le x_0.
 \end{cases}
\]
By dominated convergence, the limits of $h$ exist and
\[
 h(+\infty)=b,\qquad h(x_0^{+})=b+\sum_{u>x_0}b_u,\qquad h(x_0)-h(x_0^{+})=b_{x_0}.
\]
Hence \eqref{eq:summable} holds, \eqref{eq:jump-sum} is the definition of $h$, and $\|h\|=|b|+\sum_u|b_u|$.

(ii) For $h+h'$ we have $(h+h')(+\infty)=h(+\infty)+h'(+\infty)$ and $(h+h')(u)-(h+h')(u^{+})=h(u)-h(u^{+})+h'(u)-h'(u^{+})$. Hence \eqref{eq:summable} and \eqref{eq:jump-sum} hold for $h+h'$, and $\|h+h'\|\le\|h\|+\|h'\|$. For $hh'$ write $b=h(+\infty)$, $b_u=h(u)-h(u^{+})$, $b'=h'(+\infty)$ and $b'_v=h'(v)-h'(v^{+})$, so that $h(x)=b+\sum_{u\ge x}b_u$ and $h'(x)=b'+\sum_{v\ge x}b'_v$ by \eqref{eq:jump-sum}. Since $u\ge x$ and $v\ge x$ hold if and only if $\min(u,v)\ge x$, multiplying the two absolutely convergent series and collecting the terms with the same value of $\min(u,v)$ gives
\[
 h(x)h'(x)=bb'+\sum_{x'\ge x}\Big(bb'_{x'}+b_{x'}b'+\sum_{\min(u,v)=x'}b_ub'_v\Big),
\]
and the sum of the absolute values of the coefficients is at most $|b|\sum_v|b'_v|+|b'|\sum_u|b_u|+\sum_u|b_u|\sum_v|b'_v|$. By (i), $hh'\in\Ab$ and
\[
 \|hh'\|\le|bb'|+|b|\sum_v|b'_v|+|b'|\sum_u|b_u|+\sum_u|b_u|\sum_v|b'_v|=\|h\|\,\|h'\|.
\]

(iii) Write $b^{(i)}=h^{(i)}(+\infty)$ and $b^{(i)}_u=h^{(i)}(u)-h^{(i)}(u^{+})$, so that $h^{(i)}(x)=b^{(i)}+\sum_{u\ge x}b^{(i)}_u$ by \eqref{eq:jump-sum}. Since $\sum_i|b^{(i)}|+\sum_i\sum_u|b^{(i)}_u|=\sum_i\|h^{(i)}\|<\infty$, the sums $b=\sum_ib^{(i)}$ and $b_u=\sum_ib^{(i)}_u$ converge absolutely, $|b|+\sum_u|b_u|\le\sum_i\|h^{(i)}\|$, and rearranging the absolutely convergent double series gives $h(x)=\sum_ih^{(i)}(x)=b+\sum_{u\ge x}b_u$. By (i), $h\in\Ab$ with $\|h\|\le\sum_i\|h^{(i)}\|$,
\[
 h(+\infty)=b=\sum_ih^{(i)}(+\infty),\qquad
 h(x_0^{+})=b+\sum_{u>x_0}b_u=\sum_i\Big(b^{(i)}+\sum_{u>x_0}b^{(i)}_u\Big)=\sum_ih^{(i)}(x_0^{+}).
\]
The convergence is uniform since $|h^{(i)}(x)|\le\|h^{(i)}\|$.
\end{proof}

By \eqref{eq:jump-sum} and \cref{lem:step}~(i), the map $h\mapsto\big(h(+\infty),(h(u)-h(u^{+}))_{u\in\QQ}\big)$ is a bijection of $\Ab$ onto $\CC\oplus\ell^1(\QQ)$, and $\|h\|$ is the norm $|b|+\sum_u|b_u|$ of the image $(b,(b_u)_u)$. Hence $\Ab$ is a Banach space, and by \cref{lem:step}~(ii) a commutative Banach algebra with unit. For a set $E\subset\ZZ\tau+\ZZ$ we write $\one_E$ for its indicator function.

\begin{lemma}\label{lem:indicator}
Let $\lambda=m\tau+n\in\ZZ\tau+\ZZ$ with $\lambda\ne0$. For $x\in\QQ$,
\[
 \one_{C_x}(\lambda)=
 \begin{cases}
  1 & \text{if } m>0 \text{ and } x\le-n/m,\\
  1 & \text{if } m<0 \text{ and } x>-n/m,\\
  1 & \text{if } m=0 \text{ and } n<0,\\
  0 & \text{otherwise}.
 \end{cases}
\]
The function $x\mapsto\one_{C_x}(\lambda)$ belongs to $\Ab$ and has norm at most $2$. For every $x_0\in\QQ$,
\begin{equation}\label{eq:indicator-limits}
 \lim_{x\to x_0^{+}}\one_{C_x}(\lambda)=\one_{C_{x_0}^+}(\lambda),\text{ and }
 \lim_{x\to+\infty}\one_{C_x}(\lambda)=\one_{-P}(\lambda).
\end{equation}
\end{lemma}

\begin{proof}
By definition, $\lambda\in C_x$ if and only if $n+xm<0$, or $n+xm=0$ and $m>0$. For $m>0$ this means $x\le-n/m$, for $m<0$ it means $x>-n/m$, and for $m=0$ it means $n<0$. This proves the formula. In the notation of \cref{lem:step}~(i), the function $x\mapsto\one_{C_x}(\lambda)$ has $(b,b_{-n/m})=(0,1)$ if $m>0$ and $(b,b_{-n/m})=(1,-1)$ if $m<0$, with all other $b_u=0$, and it is constant if $m=0$. Hence it belongs to $\Ab$ with norm at most $2$. By the formula,
\[
 \lim_{x\to x_0^{+}}\one_{C_x}(\lambda)=
 \begin{cases}
  1 & \text{if } m>0 \text{ and } x_0<-n/m,\\
  1 & \text{if } m<0 \text{ and } x_0\ge-n/m,\\
  1 & \text{if } m=0 \text{ and } n<0,\\
  0 & \text{otherwise},
 \end{cases}
 \text{ and }
 \lim_{x\to+\infty}\one_{C_x}(\lambda)=
 \begin{cases}
  1 & \text{if } m<0,\\
  1 & \text{if } m=0 \text{ and } n<0,\\
  0 & \text{otherwise}.
 \end{cases}
\]
The first case distinction is the definition of $C_{x_0}^+$, since $n+x_0m<0$ means $x_0<-n/m$ for $m>0$ and $x_0>-n/m$ for $m<0$, and $n+x_0m=0$ means $x_0=-n/m$. The second one is the definition of $-P$. This proves \eqref{eq:indicator-limits}.
\end{proof}

\subsection{Sums over lattice intervals}

Let $\Bbbk=(k_1,\dots,k_r)$ be an admissible index. The entries $k_i$ with $i<r$ may be equal to $2$, and these entries need an extra argument. In \eqref{eq:GC} a point with exponent $2$ has the coefficient $\lambda^{-2}$. The points with exponent $2$ which occur for some $x\in\QQ$ are not contained in a strip of finite width, and the sum of $|\lambda|^{-2}$ over them diverges. So \cref{lem:step}~(iii) cannot be applied to \eqref{eq:GC} term by term. In this subsection we treat the sum over consecutive points with exponent $2$ between two points $\alpha$ and $\beta$. This is the elementary symmetric sum $e_N^C(\alpha,\beta)$ defined below. The main step is \eqref{eq:p-j}. It writes a sum over a lattice interval as the difference of two translates of a sum over $C$. The terms of this difference are $O(|\eta|^{-3})$.

Let $C$ be a positive set of the form \eqref{eq:lex}, and let $\alpha,\beta\in\ZZ\tau+\ZZ$. For integers $j\ge1$ and $N\ge0$ put
\[
 p_j^C(\alpha,\beta)=\sum_{\alpha\prec_C\lambda\prec_C\beta}\frac{1}{\lambda^{2j}},\qquad
 e_N^C(\alpha,\beta)=\sum_{\alpha\prec_C\lambda_1\prec_C\cdots\prec_C\lambda_N\prec_C\beta}\frac{1}{\lambda_1^2\cdots\lambda_N^2}
\]
if $0\preceq_C\alpha\prec_C\beta$, with $e_0^C(\alpha,\beta)=1$, and put $e_N^C(\alpha,\beta)=0$ otherwise. By \eqref{eq:interval-strip} and \cref{lem:lattice}~(iii) these series converge absolutely. They are the power sums and the elementary symmetric sums of the numbers $\lambda^{-2}$ with $\alpha\prec_C\lambda\prec_C\beta$.

\begin{lemma}\label{lem:blocks}
Let $\alpha,\beta\in\ZZ\tau+\ZZ$ with $\beta\ne0$, and let $N\ge0$. Then the function $h(x)=e_N^{C_x}(\alpha,\beta)$ on $\QQ$ belongs to $\Ab$, and
\[
 \|h\|\le4\big(B\log(3+|\alpha|+|\beta|)\big)^N
\]
with a constant $B\ge1$ depending only on $\tau$. For $x_0\in\QQ$ we have $h(x_0^{+})=e_N^{C_{x_0}^+}(\alpha,\beta)$ and $h(+\infty)=e_N^{-P}(\alpha,\beta)$.
\end{lemma}

\begin{proof}
Let $C$ be one of the sets $C_x$, $C_x^+$ and $-P$, and assume $0\preceq_C\alpha\prec_C\beta$. First, $e_N^C(\alpha,\beta)$ is a polynomial in the power sums. Let $\lambda$ run over finitely many nonzero numbers. The elementary symmetric functions and the power sums of the numbers $\lambda^{-2}$ are related by Newton's identities \cite{Mac}. In terms of generating functions, these identities read
\[
 \prod_\lambda\big(1+Y\lambda^{-2}\big)=\exp\Big(\sum_{j\ge1}\frac{(-1)^{j-1}Y^j}{j}\sum_\lambda\lambda^{-2j}\Big)
\]
as power series in $Y$. Comparing the coefficients of $Y^N$ gives
\begin{equation}\label{eq:newton}
 e_N^C(\alpha,\beta)=\sum_{\nu}\ \prod_{j=1}^N\frac{1}{\nu_j!}\Big(\frac{(-1)^{j-1}p_j^C(\alpha,\beta)}{j}\Big)^{\nu_j},
\end{equation}
where $\nu=(\nu_1,\dots,\nu_N)$ runs over $\ZZ_{\ge0}^N$ with $\nu_1+2\nu_2+\cdots+N\nu_N=N$. The set $\{\lambda:\alpha\prec_C\lambda\prec_C\beta\}$ may be infinite. By absolute convergence, \eqref{eq:newton} also holds for this set.

Next, the power sums are written as series in which $C$ enters only through indicator functions. For $j\ge1$,
\begin{equation}\label{eq:p-j}
 p_j^C(\alpha,\beta)=\sum_{\eta\in C}\Big(\frac{1}{(\alpha+\eta)^{2j}}-\frac{1}{(\beta+\eta)^{2j}}\Big)-\frac{1}{\beta^{2j}}\,.
\end{equation}
Here $0\preceq_C\alpha\prec_C\alpha+\eta$ and $0\prec_C\beta+\eta$, so no denominator vanishes. For $j\ge2$ the sums over $\alpha\prec_C\lambda$ and over $\beta\preceq_C\lambda$ below converge absolutely by \cref{lem:lattice}~(i). Since $\{\lambda:\beta\preceq_C\lambda\}\subset\{\lambda:\alpha\prec_C\lambda\}$, the substitutions $\lambda=\alpha+\eta$ and $\lambda=\beta+\eta$ with $\eta\in C$ give
\[
 p_j^C(\alpha,\beta)=\sum_{\alpha\prec_C\lambda}\frac{1}{\lambda^{2j}}-\sum_{\beta\preceq_C\lambda}\frac{1}{\lambda^{2j}}
 =\sum_{\eta\in C}\frac{1}{(\alpha+\eta)^{2j}}-\sum_{\eta\in C\cup\{0\}}\frac{1}{(\beta+\eta)^{2j}}\,,
\]
which is \eqref{eq:p-j}. For $j=1$ the two sums over $\lambda$ do not converge absolutely. In this case let $0<\varepsilon\le1$ and put $f_\varepsilon(z)=z^{-2}|z|^{-\varepsilon}$ for $z\ne0$. Since $\sum_{\lambda\ne0}|\lambda|^{-2-\varepsilon}<\infty$, the substitutions $\lambda=\alpha+\eta$ and $\lambda=\beta+\eta$ give
\[
 \sum_{\alpha\prec_C\lambda\prec_C\beta}f_\varepsilon(\lambda)
 =\sum_{\eta\in C}\big(f_\varepsilon(\alpha+\eta)-f_\varepsilon(\beta+\eta)\big)-f_\varepsilon(\beta).
\]
Put $A=3+|\alpha|+|\beta|$. Since $f_\varepsilon(z)=|z|^{-2-\varepsilon}e^{-2i\arg z}$, the derivative of $f_\varepsilon$ at $z$ has operator norm $(2+\varepsilon)|z|^{-3-\varepsilon}\le3|z|^{-3}$ for $|z|\ge1$ and $0\le\varepsilon\le1$. Let $|\eta|\ge2A$. Every point $z$ of the line segment from $\alpha+\eta$ to $\beta+\eta$ satisfies $|z|\ge|\eta|-\max(|\alpha|,|\beta|)\ge|\eta|/2\ge1$. Since $|\beta-\alpha|\le A$, the mean value inequality gives
\[
 |f_\varepsilon(\alpha+\eta)-f_\varepsilon(\beta+\eta)|\le|\beta-\alpha|\cdot3\Big(\frac{|\eta|}{2}\Big)^{-3}\le\frac{24A}{|\eta|^3}\qquad(|\eta|\ge2A,\ 0\le\varepsilon\le1).
\]
On the left side, $|f_\varepsilon(\lambda)|\le\max(1,\delta^{-1})|\lambda|^{-2}$, which is summable over the strip in \eqref{eq:interval-strip} by \cref{lem:lattice}~(iii). Letting $\varepsilon\to0^{+}$ on both sides by dominated convergence, where the finitely many terms with $|\eta|<2A$ converge individually, proves \eqref{eq:p-j} for $j=1$. For $|\eta|<2A$ with $\eta\notin\{-\alpha,-\beta\}$ we have $0<|\alpha+\eta|<3A$ and $0<|\beta+\eta|<3A$, and $\eta\mapsto\alpha+\eta$ and $\eta\mapsto\beta+\eta$ are injective. Since the disc $|\lambda|<3A$ lies in a strip of width $6A$, \cref{lem:lattice}~(iii) and~(ii) give
\begin{equation}\label{eq:p-1-bound}
 \sum_{\eta\notin\{0,-\alpha,-\beta\}}\Big|\frac{1}{(\alpha+\eta)^2}-\frac{1}{(\beta+\eta)^2}\Big|
 \le\sum_{0<|\lambda|<3A}\frac{2}{|\lambda|^2}+\sum_{|\eta|\ge2A}\frac{24A}{|\eta|^3}
 \ll\log(3+6A)+\frac{24A}{2A}\ll\log A.
\end{equation}

Now let $x\in\QQ$, and define $h_0(x)=\one_{C_x\cup\{0\}}(\alpha)\,\one_{C_x}(\beta-\alpha)$, and for $j\ge1$ define
\[
 h_j(x)=\sum_{\eta\notin\{0,-\alpha,-\beta\}}\one_{C_x}(\eta)\Big(\frac{1}{(\alpha+\eta)^{2j}}-\frac{1}{(\beta+\eta)^{2j}}\Big)-\frac{1}{\beta^{2j}}.
\]
Every indicator function here is an element of $\Ab$ of norm at most $2$: by \cref{lem:indicator} if its argument is nonzero, and because $\one_{C_x\cup\{0\}}(0)=1$ and $\one_{C_x}(0)=0$ do not depend on $x$ otherwise. Hence $h_0\in\Ab$ with $\|h_0\|\le4$ by \cref{lem:step}~(ii), and $h_j\in\Ab$ by \cref{lem:step}~(iii) as soon as its coefficients are absolutely summable, with
\[
 \|h_j\|\le2\sum_{\eta\notin\{0,-\alpha,-\beta\}}\Big|\frac{1}{(\alpha+\eta)^{2j}}-\frac{1}{(\beta+\eta)^{2j}}\Big|+\frac{1}{|\beta|^{2j}}\,.
\] For $j=1$ this holds by \eqref{eq:p-1-bound}, which together with $|\beta|^{-2}\le\delta^{-2}$ gives $\|h_1\|\le B\log A$ for a constant $B\ge1$ depending only on $\tau$. For $j\ge2$ the bound $|\lambda|\ge\delta$ gives
\[
 \|h_j\|\le4\sum_{\lambda\ne0}|\lambda|^{-2j}+|\beta|^{-2j}\le\delta^{-2j}\Big(4\delta^4\sum_{\lambda\ne0}|\lambda|^{-4}+1\Big).
\]
We may assume $B\ge\delta^{-2}\big(4\delta^4\sum_{\lambda\ne0}|\lambda|^{-4}+1\big)$. Since $4\delta^4\sum_{\lambda\ne0}|\lambda|^{-4}+1\ge1$ and $\log A\ge1$, this gives $\|h_j\|\le B^j\le B^j\log A$ for $2\le j\le N$. By definition, $h_0(x)=1$ if $0\preceq_{C_x}\alpha\prec_{C_x}\beta$, and $h_0(x)=0$ otherwise. In the first case $\alpha\in C_x\cup\{0\}$ and $\beta\in C_x$. Since $0\notin C_x$ and $C_x\cap(-C_x)=\emptyset$, this gives $-\alpha,-\beta\notin C_x$, so the sum over $\eta\in C_x$ in \eqref{eq:p-j} contains no term with $\eta\in\{0,-\alpha,-\beta\}$, and \eqref{eq:p-j} gives $h_j(x)=p_j^{C_x}(\alpha,\beta)$. By \eqref{eq:newton},
\[
 h=h_0\cdot\sum_{\nu}\ \prod_{j=1}^N\frac{1}{\nu_j!}\Big(\frac{(-1)^{j-1}h_j}{j}\Big)^{\nu_j}.
\]
By \cref{lem:step}~(ii), together with $\|h_0\|\le4$, $\sum_jj\nu_j=N$ and $\sum_j\nu_j\le N$, this gives
\[
 \|h\|\le4\sum_\nu\prod_{j=1}^N\frac{1}{\nu_j!}\Big(\frac{B^j\log A}{j}\Big)^{\nu_j}
 \le4(B\log A)^N\sum_\nu\prod_{j=1}^N\frac{1}{j^{\nu_j}\nu_j!}=4(B\log A)^N,
\]
since $\sum_\nu\prod_j(j^{\nu_j}\nu_j!)^{-1}$ is the coefficient of $Y^N$ in $\exp\big(\sum_{j\ge1}Y^j/j\big)=(1-Y)^{-1}$, which is $1$.

Finally, let $C'$ be $C_{x_0}^+$ or $-P$. For $0\le j\le N$ let $h_j^{C'}$ be the number defined like $h_j(x)$ with $C_x$ replaced by $C'$. By \cref{lem:indicator} and \cref{lem:step}~(iii), we have
\[
 h_j(x_0^{+})=h_j^{C_{x_0}^+},\qquad h_j(+\infty)=h_j^{-P}\qquad(0\le j\le N).
\]
The function $h$ is the polynomial expression above in $h_0$, $h_1$, \dots, $h_N$. Hence $h(x_0^{+})$ is the same expression in $h_0^{C_{x_0}^+}$, \dots, $h_N^{C_{x_0}^+}$. Likewise $h(+\infty)$ is the same expression in $h_0^{-P}$, \dots, $h_N^{-P}$. If $0\preceq_{C'}\alpha\prec_{C'}\beta$, then $h_0^{C'}=1$. Moreover $-\alpha,-\beta\notin C'$, since $0\notin C'$ and $C'\cap(-C')=\emptyset$. Hence $h_j^{C'}=p_j^{C'}(\alpha,\beta)$ by \eqref{eq:p-j}. Then the expression equals $e_N^{C'}(\alpha,\beta)$ by \eqref{eq:newton}. Otherwise $h_0^{C'}=0=e_N^{C'}(\alpha,\beta)$. This proves $h(x_0^{+})=e_N^{C_{x_0}^+}(\alpha,\beta)$ and $h(+\infty)=e_N^{-P}(\alpha,\beta)$.
\end{proof}

\subsection{Proof of \texorpdfstring{\cref{thm:transformation}}{Theorem 1.5}}

We now treat indices with entries equal to $2$ before the last entry. Since $k_r\ge3$, these entries form blocks between entries which are at least $3$. By \cref{lem:blocks}, the sum over a block between $\alpha$ and $\beta$ is a function of $x$ in $\Ab$ with norm at most $4(B\log(3+|\alpha|+|\beta|))^N$. The remaining points carry exponents which are at least $3$, so the sum over them converges by \cref{lem:lattice}.

\begin{proposition}\label{prop:jumps}
Let $\Bbbk$ be an admissible index. Then the function $h(x)=G_\Bbbk^{C_x}(\tau)$ on $\QQ$ belongs to $\Ab$, and for $x_0\in\QQ$ we have $h(x_0^{+})=G_\Bbbk^{C_{x_0}^+}(\tau)$ and $h(+\infty)=G_\Bbbk^{-P}(\tau)$.
\end{proposition}

\begin{proof}
Write
\(
 \Bbbk=(\{2\}^{N_1},l_1,\{2\}^{N_2},l_2,\dots,\{2\}^{N_t},l_t)
\)
with $N_i\ge0$ and $l_i\ge3$, where $\{2\}^N$ stands for $N$ entries equal to $2$. This decomposition exists and is unique since $k_r\ge3$. Let $C$ be one of the sets $C_x$, $C_x^+$ and $-P$. In \eqref{eq:GC}, the exponents $l_1,\dots,l_t$ belong to points $\lambda'_1\prec_C\cdots\prec_C\lambda'_t$, and between $\lambda'_{i-1}$ and $\lambda'_i$ there are exactly $N_i$ points with exponent $2$, where $\lambda'_0=0$. Summing first over these points, which is allowed by \cref{lem:convergence}, we obtain
\begin{equation}\label{eq:anchors}
 G_\Bbbk^C(\tau)=\sum_{\lambda'_1,\dots,\lambda'_t\ne0}\ \prod_{i=1}^t\frac{e_{N_i}^C(\lambda'_{i-1},\lambda'_i)}{(\lambda'_i)^{l_i}}\,,
\end{equation}
where the product vanishes unless $0\prec_C\lambda'_1\prec_C\cdots\prec_C\lambda'_t$.

Now let $C=C_x$ and fix $\lambda'_1,\dots,\lambda'_t$. By \cref{lem:blocks}, each $e_{N_i}^{C_x}(\lambda'_{i-1},\lambda'_i)$ is a function of $x$ in $\Ab$ of norm at most $4(B\log(3+|\lambda'_{i-1}|+|\lambda'_i|))^{N_i}$. Since $3+|\lambda'_{i-1}|+|\lambda'_i|\le(3+|\lambda'_{i-1}|)(3+|\lambda'_i|)$ and both logarithms on the right are at least $1$,
\[
 \log(3+|\lambda'_{i-1}|+|\lambda'_i|)\le\log(3+|\lambda'_{i-1}|)+\log(3+|\lambda'_i|)\le2\log(3+|\lambda'_{i-1}|)\log(3+|\lambda'_i|).
\]
Put $N_{t+1}=0$. Since $\lambda'_0=0$, $B\ge1$ and $N_1+\cdots+N_t\le r$, this gives
\begin{align*}
 \prod_{i=1}^t\big(B\log(3+|\lambda'_{i-1}|+|\lambda'_i|)\big)^{N_i}
 &\le(2B)^r(\log3)^{N_1}\prod_{i=1}^t\big(\log(3+|\lambda'_i|)\big)^{N_i+N_{i+1}}\\
 &\le(2B)^r\prod_{i=1}^t\big(\log(3+|\lambda'_i|)\big)^{2r}.
\end{align*}
In the last step we used $\log3\le\log(3+|\lambda'_1|)$. We also used that the exponents $2N_1+N_2$ and $N_i+N_{i+1}$ are at most $2r$ and that every logarithm is at least $1$. By \cref{lem:step}~(ii) and $t\le r$, the term of \eqref{eq:anchors} has norm at most
\[
 4^t(2B)^r\prod_{i=1}^t\frac{(\log(3+|\lambda'_i|))^{2r}}{|\lambda'_i|^{l_i}}
 \le(8B)^r\prod_{i=1}^t\frac{(\log(3+|\lambda'_i|))^{2r}}{|\lambda'_i|^{l_i}}\,.
\]
By \cref{lem:lattice}~(i) and $l_i\ge3$,
\[
 \sum_{\lambda'_1,\dots,\lambda'_t\ne0}\ (8B)^r\prod_{i=1}^t\frac{(\log(3+|\lambda'_i|))^{2r}}{|\lambda'_i|^{l_i}}
 =(8B)^r\prod_{i=1}^t\ \sum_{\lambda\ne0}\frac{(\log(3+|\lambda|))^{2r}}{|\lambda|^{l_i}}<\infty.
\]
Hence \cref{lem:step}~(iii) applies to \eqref{eq:anchors} with $C=C_x$. It gives $h\in\Ab$, and the limits of $h$ are the sums of the limits of the terms. By \cref{lem:blocks},
\begin{align*}
 h(x_0^{+})&=\sum_{\lambda'_1,\dots,\lambda'_t\ne0}\ \prod_{i=1}^t\frac{e_{N_i}^{C_{x_0}^+}(\lambda'_{i-1},\lambda'_i)}{(\lambda'_i)^{l_i}}=G_\Bbbk^{C_{x_0}^+}(\tau),
 \end{align*}
 and
 \begin{align*}
 h(+\infty)&=\sum_{\lambda'_1,\dots,\lambda'_t\ne0}\ \prod_{i=1}^t\frac{e_{N_i}^{-P}(\lambda'_{i-1},\lambda'_i)}{(\lambda'_i)^{l_i}}=G_\Bbbk^{-P}(\tau).\qedhere
\end{align*}
\end{proof}

\begin{proof}[Proof of \cref{thm:transformation}]
Write $F$ as a linear combination of the $G_\Bbbk$ with real coefficients and admissible indices $\Bbbk$ of weight $k$. Let $H(x)$ be the same linear combination of the functions $G_\Bbbk^{C_x}(\tau)$, $x\in\QQ$. By \cref{prop:jumps} and \cref{lem:step}~(ii), $H$ belongs to $\Ab$. Since the coefficients are real, $F^\flat$ is the same linear combination of the $G_\Bbbk^\flat$. Let $c'\ge1$ and $d'$ be coprime. By \cref{prop:jumps,lem:change},
\[
 H(-d'/c')=(F|_k\gamma_{c',d'})(\tau),\qquad H\big((-d'/c')^{+}\big)=(F^\flat|_k\gamma_{c',d'})(\tau),\qquad H(+\infty)=(-1)^kF(\tau).
\]
Every rational number has a unique representation $-d'/c'$ with coprime $c'\ge1$ and $d'$, and $\gamma$ may serve as $\gamma_{c,d}$. Therefore \eqref{eq:jump-sum} with $x=-d/c$ gives
\[
 (F|_k\gamma)(\tau)=(-1)^kF(\tau)+\sum_{-d'/c'\ge-d/c}\big((F-F^\flat)|_k\gamma_{c',d'}\big)(\tau),
\]
and \eqref{eq:summable} shows the series converges absolutely. The condition $-d'/c'\ge-d/c$ is equivalent to $d'/c'\le d/c$. Finally, if two matrices $\gamma,\gamma'\in\SL_2(\ZZ)$ have the same bottom row, then $\gamma'\gamma^{-1}$ is a power of $T$. Since $F$ and $F^\flat$ have period $1$, the terms do not depend on the choice of $\gamma_{c',d'}$.
\end{proof}

\subsection{Proof of \texorpdfstring{\cref{thm:main}}{Theorem 1.3}}

\begin{proof}[Proof of \cref{thm:main}]
(i) Let $k$ be even. If all $a_n$ are real, then $F^\flat=F$ by \eqref{eq:flat-fourier}, and \eqref{eq:S-formula} gives $F|_kS=F$. Since $F$ is holomorphic on $\HH$, has period $1$ and is bounded on $\{\Ima\tau\ge1\}$ by \cref{thm:fourier}, we obtain $F\in M_k$.

Conversely, let $F\in M_k$. Then $F^\flat\in M_k$. Indeed, $F^\flat$ is holomorphic, has period $1$ and is bounded on $\{\Ima\tau\ge1\}$ by \eqref{eq:flat-fourier}. Moreover, since $k$ is even and $F(-1/z)=z^kF(z)$ for $z\in\HH$,
\[
 F^\flat(-1/\tau)=\overline{F(1/\bar\tau)}=\overline{F\big(-1/(-\bar\tau)\big)}=\overline{(-\bar\tau)^kF(-\bar\tau)}=(-\tau)^k\,\overline{F(-\bar\tau)}=\tau^kF^\flat(\tau).
\]
Hence $(F-F^\flat)|_k\gamma_{c,d}=F-F^\flat$ for all $(c,d)$, and \eqref{eq:S-formula} together with $F|_kS=F$ gives
\[
 0=\sum_{\substack{c\ge1,\ d\le0\\ \gcd(c,d)=1}}\big(F(\tau)-F^\flat(\tau)\big).
\]
This is an absolutely convergent series with infinitely many equal terms, so $F(\tau)=F^\flat(\tau)$ for all $\tau\in\HH$. By \eqref{eq:flat-fourier}, all $a_n$ are real.

(ii) Let $k$ be odd and $a_n\in i\RR$ for all $n$. Then $F^\flat=F$ by \eqref{eq:flat-fourier}, and \eqref{eq:S-formula} gives $F|_kS=-F$. Hence $F|_kS^2=(-F)|_kS=F$. On the other hand $S^2=-I$, so $F|_kS^2=(-1)^kF=-F$. Therefore $F=0$.
\end{proof}

\begin{remark}\label{rem:G22}
The assumption that the indices are admissible cannot be omitted in \cref{thm:complex,thm:main}. By \cite[Example~2.6]{BT},
\[
 G_{2,2}(\tau)=\zeta(2,2)+3\zeta(2)g_2(q)+g_{2,2}(q).
\]
By \eqref{eq:g}, the $q$-series $g_2(q)=-4\pi^2\sum_{n\ge1}\sigma_1(n)q^n$ and $g_{2,2}(q)=16\pi^4(q^3+\cdots)$ have real coefficients, so all Fourier coefficients of $G_{2,2}$ are real. The first two Fourier coefficients of $G_{2,2}$ are $\zeta(2,2)=\pi^4/120$ and $3\zeta(2)\cdot(-4\pi^2)=-2\pi^4$, while those of $G_4$ are $\zeta(4)=\pi^4/90$ and $8\pi^4/3$. Since $M_4=\CC G_4$, this shows $G_{2,2}\notin M_4$. In the proof of \cref{thm:main}, the assumption $k_r\ge3$ enters through \cref{lem:convergence}. For $k_r=2$ the series \eqref{eq:GC} does not converge absolutely. Then the change of the order of summation in the proof of \cref{lem:change} is not justified.
\end{remark}

\subsection{Proof of \texorpdfstring{\cref{thm:complex}}{Theorem 1.1} and \texorpdfstring{\cref{cor:rational}}{Corollary 1.4}}\label{subsec:complex}

Let $\Bbbk_1,\dots,\Bbbk_N$ be all admissible indices of weight $k$. Let $F\in\Ef_k^{\mathrm{adm}}$, and write $F=\sum_{\nu=1}^Nc_\nu G_{\Bbbk_\nu}$ with $c_\nu\in\CC$. The proof of \cref{thm:transformation} uses that $F^\flat$ is the same linear combination of the $G_{\Bbbk_\nu}^\flat$. For complex coefficients this fails, since $F\mapsto F^\flat$ is conjugate linear. The next lemma compares $G_\Bbbk^\flat$ and $G_\Bbbk$ by the weights of the terms in their Fourier expansions. It is used to replace the coefficients $c_\nu$ by real numbers.

\begin{lemma}\label{lem:flat-lower}
Let $\Bbbk$ be an index of weight $k$. Then
\[
 G_\Bbbk^\flat-G_\Bbbk\in\Ef_0+\Ef_1+\cdots+\Ef_{k-1}.
\]
\end{lemma}

\begin{proof}
In this proof $F\mapsto F^\flat$ is formed with the weight $k$ of $G_\Bbbk$, also for the terms of its Fourier expansion. Let $\Bbbk'$ be an index of weight $k'$. By \eqref{eq:g}, the Fourier coefficients of $g_{\Bbbk'}(q)$ are rational multiples of $(-2\pi i)^{k'}$. Since $\overline{(-2\pi i)^{k'}}=(-1)^{k'}(-2\pi i)^{k'}$, \eqref{eq:flat-fourier} gives
\[
 g_{\Bbbk'}^\flat=(-1)^{k+k'}g_{\Bbbk'}.
\]
In particular $g_\Bbbk^\flat=g_\Bbbk$. By \cref{thm:fourier}, the constant term $\zeta(\Bbbk)$ and the coefficients $c_{\Bbbk'}$ are real. Applying $F\mapsto F^\flat$ to the Fourier expansion in \cref{thm:fourier} therefore gives
\[
 G_\Bbbk^\flat-G_\Bbbk=\big((-1)^k-1\big)\zeta(\Bbbk)+\sum_{\Bbbk'}\big((-1)^{k+k'}-1\big)c_{\Bbbk'}\,g_{\Bbbk'}(q),
\]
where $k'$ denotes the weight of $\Bbbk'$. The term $g_\Bbbk(q)$ has cancelled, and every $\Bbbk'$ in the sum has weight less than $k$. It remains to show that
\[
 g_{\Bbbk'}(q)\in\Ef_0+\Ef_1+\cdots+\Ef_{k'}
\]
for every index $\Bbbk'$ of weight $k'$. This follows by induction on $k'$. By \cref{thm:fourier} for $\Bbbk'$,
\[
 g_{\Bbbk'}(q)=G_{\Bbbk'}-\zeta(\Bbbk')-\sum_{\Bbbk''}c_{\Bbbk''}\,g_{\Bbbk''}(q),
\]
where $\Bbbk''$ runs over finitely many indices of weight less than $k'$. For $k'=2$ the sum is empty. Since $G_{\Bbbk'}\in\Ef_{k'}$ and $\zeta(\Bbbk')\in\Ef_0$, the induction hypothesis gives the assertion.
\end{proof}

\begin{proof}[Proof of \cref{thm:complex}]
Let $F\in\Ef_k^{\RR,\mathrm{adm}}$, and write
\[
 F=\sum_{\nu=1}^Nc_\nu G_{\Bbbk_\nu}
\]
with $c_\nu\in\CC$, where $\Bbbk_1,\dots,\Bbbk_N$ are all admissible indices of weight $k$. Since the Fourier coefficients of $F$ are real, \eqref{eq:flat-fourier} gives $F^\flat=(-1)^kF$. Since $F\mapsto F^\flat$ is conjugate linear, \cref{lem:flat-lower} gives
\[
 (-1)^kF-\sum_{\nu=1}^N\overline{c_\nu}\,G_{\Bbbk_\nu}=\sum_{\nu=1}^N\overline{c_\nu}\big(G_{\Bbbk_\nu}^\flat-G_{\Bbbk_\nu}\big)\in\Ef_0+\Ef_1+\cdots+\Ef_{k-1}.
\]
The left side lies in $\Ef_k$. By \cite[Main Theorem~B(i)]{BKa}, the sum $\Ef_0+\Ef_1+\Ef_2+\cdots$ is direct. Hence the left side vanishes, and
\[
 (-1)^kF=\sum_{\nu=1}^N\overline{c_\nu}\,G_{\Bbbk_\nu}.
\]
Let $k$ be even. Adding this expression to the original one gives
\[
 F=\sum_{\nu=1}^N\Rea(c_\nu)\,G_{\Bbbk_\nu}.
\]
So $F$ is a real linear combination of admissible multiple Eisenstein series of weight $k$, and its Fourier coefficients are real. By \cref{thm:main}~(i), $F\in M_k$. Let $k$ be odd. Subtracting this expression from the original one gives $F=iH$ with
\[
 H=\sum_{\nu=1}^N\Ima(c_\nu)\,G_{\Bbbk_\nu}.
\]
The Fourier coefficients of $H=-iF$ lie in $i\RR$, so \cref{thm:main}~(ii) gives $H=0$. Hence $F=0$. This proves $\Ef_k^{\RR,\mathrm{adm}}\subset M_k^\RR$.

For the reverse inclusion, recall that $M_k$ is spanned by the products $G_4^aG_6^b$ with $a,b\ge0$ and $4a+6b=k$. Multiply the absolutely convergent series \eqref{eq:mes} of the factors, and group the terms according to the order of the lattice points and their coincidences. This is the harmonic product. It expresses $G_4^aG_6^b$ as a linear combination with integer coefficients of $G_\Bbbk$ of weight $k$ whose entries are at least $4$. Hence $M_k\subset\Ef_k^{\mathrm{adm}}$, and $M_k^\RR\subset\Ef_k^{\RR,\mathrm{adm}}$. Finally,
\[
 \Ef_k^{\QQ,\mathrm{adm}}=\Ef_k^{\RR,\mathrm{adm}}\cap\QQ[[q]]=M_k^\RR\cap\QQ[[q]]=M_k^\QQ.\qedhere
\]
\end{proof}

\begin{proof}[Proof of \cref{cor:rational}]
By the proof of \cref{thm:complex}, the products $G_4^aG_6^b$ with $4a+6b=k$ are rational linear combinations of $G_1,\dots,G_N$. Choose a basis $w_1,\dots,w_s$ of the $\QQ$-span of these products. Extend it to a basis $w_1,\dots,w_s,v_1,\dots,v_t$ of the $\QQ$-span of $G_1,\dots,G_N$. By \cite[Main Theorem~B(ii)]{BKa}, rational linear combinations of multiple Eisenstein series that are linearly independent over $\QQ$ are also linearly independent over $\CC$. In particular, the functions $w_1,\dots,w_s,v_1,\dots,v_t$ are linearly independent over $\CC$. Moreover, $M_k$ is the $\CC$-span of $w_1,\dots,w_s$.

(i) Write $G_\nu=\sum_i\alpha_{\nu i}w_i+\sum_j\beta_{\nu j}v_j$ with rational $\alpha_{\nu i}$ and $\beta_{\nu j}$. For $c_1,\dots,c_N\in\CC$ we obtain
\[
 \sum_{\nu=1}^Nc_\nu G_\nu=\sum_{i=1}^s\Big(\sum_{\nu=1}^Nc_\nu\alpha_{\nu i}\Big)w_i+\sum_{j=1}^t\Big(\sum_{\nu=1}^Nc_\nu\beta_{\nu j}\Big)v_j.
\]
This function lies in $M_k$ if and only if $\sum_\nu c_\nu\beta_{\nu j}=0$ for $1\le j\le t$. These linear equations have rational coefficients. So their solutions form a vector space with a basis of rational vectors.

(ii) Let $k$ be even. Then $(2\pi i)^k$ is real, so $(2\pi i)^{-k}F$ has real Fourier coefficients if and only if $F$ has. By \cref{thm:main}~(i), this holds if and only if $F$ is a modular form. Let $k$ be odd. Then $(2\pi i)^{-k}F$ has real Fourier coefficients if and only if the Fourier coefficients of $F$ lie in $i\RR$. By \cref{thm:main}~(ii), this holds if and only if $F=0$. Since $M_k=\{0\}$ for odd $k$, this proves the first statement in (ii).

Now let $F$ be a modular form. Write $F=\sum_ia_iw_i+\sum_jb_jv_j$ with $a_i,b_j\in\QQ$. Since $F\in M_k$, we also have $F=\sum_id_iw_i$ with $d_i\in\CC$. Comparing coefficients gives $b_j=0$ for all $j$. Hence $F$ is a rational linear combination of the products $G_4^aG_6^b$. Since $\zeta(4)=\pi^4/90$ and $\zeta(6)=\pi^6/945$, the Fourier expansions of $G_4$ and $G_6$ give
\[
 (2\pi i)^{-4}G_4=\frac{1}{1440}+\frac16\sum_{n\ge1}\sigma_3(n)q^n,\qquad
 (2\pi i)^{-6}G_6=-\frac{1}{60480}+\frac{1}{120}\sum_{n\ge1}\sigma_5(n)q^n.
\]
Hence $(2\pi i)^{-k}G_4^aG_6^b$ has rational Fourier coefficients if $4a+6b=k$, and so does $(2\pi i)^{-k}F$.
\end{proof}

\section{\texorpdfstring{$L$}{L}-functions and the proof of \texorpdfstring{\cref{thm:hecke}}{Theorem 1.8}}\label{sec:hecke}

In this section $F\in\Ef_k$, so the indices of the multiple Eisenstein series in $F$ need not be admissible. Let $F(\tau)=\sum_{n\ge0}a_nq^n$ be its Fourier expansion, see \cref{thm:fourier}, and let $L(F,s)$ and $\Lambda_F(s)$ be as in \cref{sec:intro}.

\begin{proof}[Proof of \cref{prop:analytic}]
By \cref{thm:fourier}, $F-a_0$ is a finite linear combination of $q$-series $g_{\Bbbk'}(q)$ with indices $\Bbbk'$ of weight at most $k$. Fix such an index $\Bbbk'=(k'_1,\dots,k'_r)$ of weight $k'$. For an integer $m\ge2$ and $y>0$ put
\[
 \varphi_m(y)=\frac{1}{(m-1)!}\sum_{v\ge1}v^{m-1}e^{-vy}=\frac{1}{(m-1)!}\Big(-\frac{d}{dy}\Big)^{m-1}\frac{1}{e^y-1}\,.
\]
By \eqref{eq:g},
\begin{equation}\label{eq:g-phi}
 (-2\pi i)^{-k'}g_{\Bbbk'}(e^{-y})=\sum_{0<u_1<\cdots<u_r}\varphi_{k'_1}(u_1y)\cdots\varphi_{k'_r}(u_ry).
\end{equation}
Since $1/(e^y-1)-1/y$ is holomorphic at $y=0$ and $(-d/dy)^{m-1}y^{-1}=(m-1)!\,y^{-m}$, we have
\[
 \varphi_m(y)=y^{-m}+\frac{1}{(m-1)!}\Big(-\frac{d}{dy}\Big)^{m-1}\Big(\frac{1}{e^y-1}-\frac1y\Big)=y^{-m}+O(1)\qquad(0<y\le1).
\]
For $y\ge1$ and $v\ge1$ we have $e^{-vy}\le e^{1-y}e^{-v}$, hence $\varphi_m(y)\le e^{1-y}\varphi_m(1)$. Therefore $y^m\varphi_m(y)$ is bounded on $(0,\infty)$ and tends to $1$ as $y\to0^{+}$. With \eqref{eq:g-phi} this gives, with an implied constant depending on $\Bbbk'$,
\begin{equation}\label{eq:g-bound}
 0\le(-2\pi i)^{-k'}g_{\Bbbk'}(e^{-y})\ll y^{-k'}\sum_{0<u_1<\cdots<u_r}\frac{1}{u_1^{k'_1}\cdots u_r^{k'_r}}=\zeta(\Bbbk')\,y^{-k'}\qquad(y>0).
\end{equation}
Write $(-2\pi i)^{-k'}g_{\Bbbk'}(q)=\sum_{n\ge1}c_nq^n$ with $c_n\ge0$. Then $\sum_nc_ne^{-ny}\le e^{1-y}\sum_nc_ne^{-n}$ for $y\ge1$, and Tonelli's theorem and \eqref{eq:g-bound} give
\[
 \Gamma(\sigma)\sum_{n\ge1}\frac{c_n}{n^\sigma}=\int_0^\infty\sum_{n\ge1}c_ne^{-ny}\,y^{\sigma-1}\,dy<\infty\qquad(\sigma>k').
\]
Since $k'\le k$, the series $L(F,s)$ converges absolutely for $\Rea s>k$. This proves (i).

For (ii) and (iii), write $u_j=d_1+\cdots+d_j$ with $d_1,\dots,d_r\ge1$. Then $u_1v_1+\cdots+u_rv_r=\sum_{j=1}^rd_j(v_j+\cdots+v_r)$, and \eqref{eq:g} gives
\begin{equation}\label{eq:sargos}
 L(g_{\Bbbk'},s)=\frac{(-2\pi i)^{k'}}{(k'_1-1)!\cdots(k'_r-1)!}\sum_{d_1,\dots,d_r,v_1,\dots,v_r\ge1}\frac{v_1^{k'_1-1}\cdots v_r^{k'_r-1}}{\big(\sum_{j=1}^rd_j(v_j+\cdots+v_r)\big)^s}\,.
\end{equation}
The polynomial in the denominator has positive coefficients, and
\[
 \sum_{j=1}^rd_j(v_j+\cdots+v_r)\ge\max(d_1+\cdots+d_r,\ v_1+\cdots+v_r)\qquad\text{on }[1,\infty)^{2r}.
\] In particular it tends to infinity as $|(d_1,\dots,d_r,v_1,\dots,v_r)|\to\infty$ on $[1,\infty)^{2r}$, which is the hypothesis (1.1) of \cite{Sa}. Hence Theorems~1.2 and~1.4 of Sargos \cite{Sa} apply to \eqref{eq:sargos}. They show that $L(g_{\Bbbk'},s)$ has a meromorphic continuation to $\CC$ with real poles, and that it grows at most polynomially in $|\Ima s|$ on every vertical strip for $|\Ima s|\ge1$. The exponential factor in the growth estimate of \cite[Theorem~1.4]{Sa} equals $1$ here, since all coefficients of the polynomial are positive. Together with Stirling's formula for $\Gamma(s)$, this proves (ii) and (iii).
\end{proof}

\begin{lemma}\label{lem:mellin}
\begin{enumerate}
\item For $\Rea s>k$,
\begin{equation}\label{eq:mellin}
 \Lambda_F(s)=\int_0^\infty\big(F(iy)-a_0\big)y^{s-1}\,dy.
\end{equation}
\item For $y>0$ and $\sigma>k$,
\begin{equation}\label{eq:inversion}
 F(iy)-a_0=\frac{1}{2\pi i}\int_{\sigma-i\infty}^{\sigma+i\infty}\Lambda_F(s)y^{-s}\,ds.
\end{equation}
\end{enumerate}
\end{lemma}

\begin{proof}
(i) Let $\sigma=\Rea s>k$. By \cref{prop:analytic}~(i), $\sum_n|a_n|(2\pi n)^{-\sigma}\Gamma(\sigma)<\infty$. Hence \eqref{eq:mellin} follows from $\Gamma(s)(2\pi n)^{-s}=\int_0^\infty e^{-2\pi ny}y^{s-1}\,dy$ by termwise integration.

(ii) We use $e^{-y}=\frac{1}{2\pi i}\int_{\sigma-i\infty}^{\sigma+i\infty}\Gamma(s)y^{-s}\,ds$ for $y>0$ and $\sigma>0$. For $\sigma>k$ we have $\sum_n|a_n|n^{-\sigma}\int_{-\infty}^\infty|\Gamma(\sigma+it)|\,dt<\infty$. Hence summation and integration may be interchanged in
\[
 F(iy)-a_0=\sum_{n\ge1}a_n\,\frac{1}{2\pi i}\int_{\sigma-i\infty}^{\sigma+i\infty}\Gamma(s)(2\pi ny)^{-s}\,ds
 =\frac{1}{2\pi i}\int_{\sigma-i\infty}^{\sigma+i\infty}(2\pi)^{-s}\Gamma(s)L(F,s)\,y^{-s}\,ds.
\]
This is \eqref{eq:inversion}.
\end{proof}

\begin{lemma}\label{lem:cusp}
As $y\to0^{+}$, the function $y^kF(iy)$ tends to $(-i)^ka_0$.
\end{lemma}

\begin{proof}
By linearity we may assume $F=G_\Bbbk$ with an index $\Bbbk$ of weight $k$, so that $a_0=\zeta(\Bbbk)$. By \cref{thm:fourier} with $q=e^{-2\pi y}$ and by \eqref{eq:g-bound}, $y^kg_{\Bbbk'}(e^{-2\pi y})\to0$ as $y\to0^{+}$ for every index $\Bbbk'$ of weight less than $k$, and $y^k\zeta(\Bbbk)\to0$. By \eqref{eq:g-phi},
\[
 y^kg_\Bbbk(e^{-2\pi y})=(-i)^k\sum_{0<u_1<\cdots<u_r}\ \prod_{j=1}^r\frac{(2\pi u_jy)^{k_j}\varphi_{k_j}(2\pi u_jy)}{u_j^{k_j}}\,.
\]
Each term tends to $(u_1^{k_1}\cdots u_r^{k_r})^{-1}$ and is bounded by a constant multiple of it. By dominated convergence, $y^kg_\Bbbk(e^{-2\pi y})\to(-i)^k\zeta(\Bbbk)=(-i)^ka_0$.
\end{proof}

\begin{proof}[Proof of \cref{thm:hecke}]
Put $f(y)=F(iy)-a_0$ for $y>0$. For $y\ge1$,
\begin{equation}\label{eq:f-decay}
 |f(y)|\le e^{-2\pi(y-1)}\sum_{n\ge1}|a_n|e^{-2\pi n}.
\end{equation}
Since $k$ is even, $i^k=i^{-k}=(-i)^k\in\{\pm1\}$.

Suppose first that $F\in M_k$. Since $F(i/y)=(iy)^kF(iy)$, we have
\[
 f(1/y)=(iy)^k\big(f(y)+a_0\big)-a_0=i^ky^kf(y)+i^ka_0y^k-a_0.
\]
Put $J(s)=\int_1^\infty f(y)y^{s-1}\,dy$. By \eqref{eq:f-decay}, $J$ is entire. Splitting \eqref{eq:mellin} at $y=1$ and substituting $y\mapsto1/y$ on $(0,1)$ gives, for $\Rea s>k$,
\begin{align*}
 \Lambda_F(s)&=J(s)+\int_1^\infty f(1/y)\,y^{-s-1}\,dy
 =J(s)+\int_1^\infty\big(i^ky^kf(y)+i^ka_0y^k-a_0\big)y^{-s-1}\,dy\\
 &=J(s)+i^kJ(k-s)+\frac{i^ka_0}{s-k}-\frac{a_0}{s}\,.
\end{align*}
Under $s\mapsto k-s$ the right side becomes $J(k-s)+i^kJ(s)-i^ka_0/s+a_0/(s-k)$, which is $i^k$ times the right side since $i^{2k}=1$. This proves $\Lambda_F(s)=i^k\Lambda_F(k-s)$.

Conversely, assume that $\Lambda_F(s)=i^k\Lambda_F(k-s)$. The function $\Lambda_F$ is holomorphic for $\Rea s>k$ by \cref{prop:analytic}~(i), hence for $\Rea s<0$ by the functional equation. Since its poles are real and discrete, $\Lambda_F$ has only finitely many poles, all of them in $[0,k]$. By the functional equation, $\rho$ is a pole if and only if $k-\rho$ is a pole.

The poles at $k$ and $0$ are determined by the constant term $a_0$. Put $h(y)=y^kf(y)$ for $0<y\le1$. By \cref{lem:cusp}, $h(y)\to i^ka_0$ as $y\to0^{+}$, so $h$ is bounded. For $0<\varepsilon\le1$, \eqref{eq:mellin} and the substitution $u=y^\varepsilon$, $du=\varepsilon y^{\varepsilon-1}\,dy$, on $(0,1)$ give
\[
 \varepsilon\Lambda_F(k+\varepsilon)=\varepsilon\int_1^\infty f(y)y^{k+\varepsilon-1}\,dy+\int_0^1h(y)\,\varepsilon y^{\varepsilon-1}\,dy
 =\varepsilon\int_1^\infty f(y)y^{k+\varepsilon-1}\,dy+\int_0^1h(u^{1/\varepsilon})\,du.
\]
The first term tends to $0$ as $\varepsilon\to0^{+}$ by \eqref{eq:f-decay}. In the second term, $h(u^{1/\varepsilon})\to i^ka_0$ for every $0<u<1$. Since the integrand is bounded, the integral tends to $i^ka_0$ by dominated convergence. Hence $\varepsilon\Lambda_F(k+\varepsilon)\to i^ka_0$ as $\varepsilon\to0^{+}$. A pole of order at least two at $k$ would make $\varepsilon\Lambda_F(k+\varepsilon)$ unbounded, so $\Lambda_F$ has at most a simple pole at $k$, with residue $i^ka_0$. By the functional equation, it has at most a simple pole at $0$, with residue $-a_0$.

Next we compute $F(iy)-(iy)^{-k}F(i/y)$ by Mellin inversion. Fix $\sigma>k$. By \cref{prop:analytic}~(iii), the integral of $\Lambda_F(s)y^{-s}$ over the line $\Rea s=k-\sigma$ converges absolutely, and the integrals over the horizontal segments from $k-\sigma\pm iV$ to $\sigma\pm iV$ tend to $0$ as $V\to\infty$. Moving the line of integration in \eqref{eq:inversion} to $\Rea s=k-\sigma$ therefore gives
\[
 f(y)=\frac{1}{2\pi i}\int_{k-\sigma-i\infty}^{k-\sigma+i\infty}\Lambda_F(s)y^{-s}\,ds+\sum_\rho\Res_{s=\rho}\big(\Lambda_F(s)y^{-s}\big),
\]
where $\rho$ runs over the poles of $\Lambda_F$. The substitution $s=k-s'$, which reverses the orientation of the line, the functional equation $\Lambda_F(k-s')=i^k\Lambda_F(s')$ and \eqref{eq:inversion} with $1/y$ in place of $y$ give
\begin{align*}
 \frac{1}{2\pi i}\int_{k-\sigma-i\infty}^{k-\sigma+i\infty}\Lambda_F(s)y^{-s}\,ds
 &=\frac{1}{2\pi i}\int_{\sigma-i\infty}^{\sigma+i\infty}\Lambda_F(k-s')\,y^{s'-k}\,ds'\\
 &=i^ky^{-k}\,\frac{1}{2\pi i}\int_{\sigma-i\infty}^{\sigma+i\infty}\Lambda_F(s')(1/y)^{-s'}\,ds'=i^ky^{-k}f(1/y).
\end{align*}
The residues at $k$ and $0$ are
\[
 \Res_{s=k}\big(\Lambda_F(s)y^{-s}\big)=i^ka_0y^{-k},\text{ and }
 \Res_{s=0}\big(\Lambda_F(s)y^{-s}\big)=-a_0.
\]
Since $i^ky^{-k}f(1/y)=(iy)^{-k}F(i/y)-i^ka_0y^{-k}$, the terms with $a_0$ cancel, and we obtain
\begin{equation}\label{eq:defect}
 F(iy)-(iy)^{-k}F(i/y)=\sum_{0<\rho<k}\Res_{s=\rho}\big(\Lambda_F(s)y^{-s}\big)\qquad(y>0).
\end{equation}

Put $\psi(z)=F(z)-z^{-k}F(-1/z)$ for $z\in\HH$, so that the left side of \eqref{eq:defect} is $\psi(iy)$. Let $\Omega=\CC\setminus\{-iy:y\ge0\}$, and for $z\in\Omega$ put $(-iz)^{-s}=\exp(-s\Log(-iz))$ with the principal branch of the logarithm. The right side of \eqref{eq:defect} with $y^{-s}$ replaced by $(-iz)^{-s}$ is a holomorphic function on $\Omega$, which agrees with $\psi$ on the positive imaginary axis. By the identity theorem it agrees with $\psi$ on $\HH$, and we use it to extend $\psi$ to $\Omega$. For $t>0$ we have $\Log(-it)=\log t-\pi i/2$. Let $\rho$ be a pole of order $\nu$, and let $\Lambda_F(s)=\sum_{j=1}^{\nu}b_j(s-\rho)^{-j}+O(1)$ near $\rho$, so that $b_\nu\ne0$. For $t>0$ we have
\[
 (-it)^{-s}=e^{-s(\log t-\pi i/2)}=t^{-\rho}e^{\pi i\rho/2}\sum_{n\ge0}\frac{(\pi i/2-\log t)^n}{n!}\,(s-\rho)^n.
\]
The residue of $\Lambda_F(s)(-it)^{-s}$ at $\rho$ is the coefficient of $(s-\rho)^{-1}$ in the product of the two expansions. This coefficient is the sum over $j$ of $b_j$ times the coefficient of $(s-\rho)^{j-1}$ in the second expansion. Hence
\begin{equation}\label{eq:residue-poly}
 \Res_{s=\rho}\big(\Lambda_F(s)(-it)^{-s}\big)=t^{-\rho}P_\rho(\log t), \text{ and }
 P_\rho(X)=e^{\pi i\rho/2}\sum_{j=1}^{\nu}\frac{b_j}{(j-1)!}\Big(\frac{\pi i}{2}-X\Big)^{j-1}.
\end{equation}
Thus $P_\rho$ is a polynomial of degree $\nu-1$ with nonzero leading coefficient. For $z\in\HH$ we have $\Log(-iz)=\Log z-\pi i/2$, and the same computation gives $\Res_{s=\rho}\big(\Lambda_F(s)(-iz)^{-s}\big)=z^{-\rho}P_\rho(\Log z)$. This is the description of $\psi$ given in \cref{sec:intro}.

Since $F(z+1)=F(z)$, the function $\psi$ satisfies the three-term relation
\begin{equation}\label{eq:three-term}
 \psi(z)=\psi(z+1)+(z+1)^{-k}\psi\Big(\frac{z}{z+1}\Big)\qquad(z\in\HH).
\end{equation}
The right side equals
\[
 F(z+1)-(z+1)^{-k}F\Big(\frac{-1}{z+1}\Big)+(z+1)^{-k}F\Big(\frac{z}{z+1}\Big)-z^{-k}F\Big(-\frac{z+1}{z}\Big).
\]
The two middle terms cancel since $z/(z+1)=1-1/(z+1)$, and the last term equals $-z^{-k}F(-1-1/z)=-z^{-k}F(-1/z)$. This proves \eqref{eq:three-term}. Put $U=\HH\cup\{z\in\CC:\Rea z>0\}$. For $z\in U$ the points $z$, $z+1$ and $z/(z+1)$ lie in $\Omega$, so both sides of \eqref{eq:three-term} are holomorphic on $U$. Since $U$ is connected, \eqref{eq:three-term} holds on $U$, in particular for $z=t\in(0,\infty)$.

Assume now that $\Lambda_F$ has a pole in $(0,k)$, and let $\rho_*$ be the largest one. Since the poles are symmetric with respect to $k/2$, we have $\rho_*\ge k/2\ge2$. Let $\rho\in(0,k)$ be a pole of order $\nu$. As $t\to0^{+}$ we have $(1+t)^{\rho-k}=1+(\rho-k)t+O(t^2)$ and $\log(1+t)=t+O(t^2)$, and the Taylor expansion of the polynomial $P_\rho$ gives $P_\rho(\log t-\log(1+t))=P_\rho(\log t)-tP_\rho'(\log t)+O(t^2|\log t|^{\nu-1})$. Hence
\begin{align*}
 (1+t)^{-k}&\Big(\frac{t}{1+t}\Big)^{-\rho}P_\rho\Big(\log\frac{t}{1+t}\Big)
 =t^{-\rho}(1+t)^{\rho-k}P_\rho\big(\log t-\log(1+t)\big)\\
 &=t^{-\rho}P_\rho(\log t)-t^{1-\rho}\big((k-\rho)P_\rho(\log t)+P_\rho'(\log t)\big)+O\big(t^{2-\rho}|\log t|^{\nu-1}\big).
\end{align*}
Let $m$ be the maximal order of a pole of $\Lambda_F$ in $(0,k)$. Summing over $\rho$ and using \eqref{eq:residue-poly} and \eqref{eq:three-term} with $z=t$, we find
\[
 \psi(1+t)=\psi(t)-(1+t)^{-k}\psi\Big(\frac{t}{1+t}\Big)=\sum_{0<\rho<k}t^{1-\rho}\big((k-\rho)P_\rho(\log t)+P_\rho'(\log t)\big)+O\big(t^{2-\rho_*}|\log t|^{m}\big).
\]
Let $\nu_*$ be the order of $\rho_*$, and let $\theta\ne0$ be the leading coefficient of $P_{\rho_*}$. Since $P_{\rho_*}'$ is $0$ for $\nu_*=1$ and has degree $\nu_*-2$ for $\nu_*\ge2$,
\[
 (k-\rho_*)P_{\rho_*}(\log t)+P_{\rho_*}'(\log t)=(k-\rho_*)\theta(\log t)^{\nu_*-1}+O\big(|\log t|^{\nu_*-2}\big).
\]
Let $\delta>0$ be the minimum of $1$ and of $\rho_*-\rho$ over the poles $\rho<\rho_*$ in $(0,k)$. The terms with $\rho<\rho_*$ and the error term are $O(t^{1-\rho_*+\delta}|\log t|^{m})$. Hence, for $0<t<1/e$,
\[
 \psi(1+t)=t^{1-\rho_*}(\log t)^{\nu_*-1}\big((k-\rho_*)\theta+R(t)\big),\qquad R(t)\ll|\log t|^{-1}+t^{\delta}|\log t|^{m},
\]
and $R(t)\to0$ as $t\to0^{+}$. Choose $t_0>0$ with $|R(t)|\le(k-\rho_*)|\theta|/2$ for $0<t<t_0$. Since $k-\rho_*>0$ and $1-\rho_*<0$,
\[
 |\psi(1+t)|\ge\tfrac12(k-\rho_*)|\theta|\,t^{1-\rho_*}\to\infty\qquad(t\to0^{+}).
\]
This contradicts the holomorphy of $\psi$ at $1\in\Omega$. Therefore $\Lambda_F$ has no poles in $(0,k)$. Then \eqref{eq:defect} gives $\psi(iy)=0$ for $y>0$, and $\psi=0$ on $\HH$ by the identity theorem. Hence $F|_kS=F$. Together with \cref{thm:fourier}, this shows $F\in M_k$.
\end{proof}

\begin{remark}\label{rem:G2}
The condition $k\ge4$ in \cref{thm:hecke} cannot be omitted. The Fourier expansion $G_2(\tau)=\zeta(2)-4\pi^2\sum_{n\ge1}\sigma_1(n)q^n$ and $\sum_{n\ge1}\sigma_1(n)n^{-s}=\zeta(s)\zeta(s-1)$ give
\[
 \Lambda_{G_2}(s)=-4\pi^2(2\pi)^{-s}\Gamma(s)\zeta(s)\zeta(s-1).
\]
Put $\xi(s)=\pi^{-s/2}\Gamma(s/2)\zeta(s)$. The duplication formula $\Gamma(s/2)\Gamma((s-1)/2)=2^{2-s}\sqrt\pi\,\Gamma(s-1)$ and $\Gamma(s)=(s-1)\Gamma(s-1)$ give
\[
 \xi(s)\xi(s-1)=\pi^{\frac12-s}\,2^{2-s}\sqrt\pi\,\Gamma(s-1)\zeta(s)\zeta(s-1)=\frac{4\pi(2\pi)^{-s}}{s-1}\,\Gamma(s)\zeta(s)\zeta(s-1),
\]
and therefore
\[
 \Lambda_{G_2}(s)=-\pi(s-1)\,\xi(s)\,\xi(s-1).
\]
Since $\xi(s)=\xi(1-s)$, we have $\xi(2-s)\xi(1-s)=\xi(s-1)\xi(s)$, so that $\Lambda_{G_2}(2-s)=-\pi(1-s)\xi(s)\xi(s-1)=-\Lambda_{G_2}(s)$. This is the functional equation for $k=2$. However $G_2$ is not a modular form, since $M_2=\{0\}$. In this case $\Lambda_{G_2}$ has a pole at $\rho_*=1=k/2$. Then $t^{1-\rho_*}=1$, and $\psi(1+t)$ stays bounded as $t\to0^{+}$. So the argument above gives no contradiction. Indeed $\psi(z)=\pi i/z$ for $F=G_2$, which is holomorphic at $z=1$.
\end{remark}

{\bf AI \& computational resource disclosure:} The results, their formulation, and the proof strategy are the author's own. ChatGPT~6 Astra and Claude Opus~5.5 and Fable~5.1 were used throughout as research assistants: they computed and checked explicit examples of modular forms via SageMath. These examples suggested the criterion in \cref{thm:main}. The same tools helped to simplify the proofs in \cref{subsec:complex} and to improve the English. All statements and proofs were verified by the author, who takes sole responsibility for them.

\end{document}